\documentclass[12pt, reqno]{amsart}
\usepackage{kpfonts}
\usepackage{amsmath,amsthm,amscd,amsfonts,amssymb,graphicx,color}
\usepackage[bookmarksnumbered,colorlinks,plainpages]{hyperref}

\usepackage{amsmath,amssymb}
\usepackage{graphicx}
\usepackage{amsmath, amssymb}
\usepackage{tikz}
\usepackage{cases}
\usetikzlibrary{positioning, arrows.meta}
\usepackage[a4paper,right=1in,left=1in,top=1.5in,bottom=1.5in]{geometry}
\usepackage{ragged2e}
\usepackage{float}
\hypersetup{colorlinks=true,linkcolor=black, anchorcolor=black,
	citecolor=black, urlcolor=black, filecolor=magenta, pdftoolbar=true}
    \newtheorem{theorem}{Theorem}[section]
\newtheorem{lemma}{Lemma}[section]
\newtheorem{proposition}{Proposition}[section]

\theoremstyle{definition}
\newtheorem{definition}{Definition}[section]
\newtheorem{example}{Example}[section]

\theoremstyle{remark}
\newtheorem{remark}{Remark}[section]
\numberwithin{equation}{section}

\begin{document}
\setcounter{page}{1}
\setcounter{page}{1}

\noindent 
{\textbf{\Large Interval-Valued Optimization Problems for Strongly LU-$\mathcal{E}$-Invex and Strongly LU-$\mathcal{E}$-Preinvex Functions}}\\

	\begin{center}
		{\bf Tauheed, Akhlad Iqbal and Amir Suhail  }
	\end{center}
\bigskip

\textbf{Abstract}: In this paper, we introduce and explore the concepts of  strongly LU-$\mathcal{E}$-preinvex (SLU$\mathcal{E}$P), pseudo strongly  LU-$\mathcal{E}$-preinvex (PSLU$\mathcal{E}$P) and strongly LU-$\mathcal{E}$-invex (SLU$\mathcal{E}$I) functions. To illustrate and validate these definitions, we provide several non-trivial examples. Additionally, we extend the idea of strongly-G invex sets to the context of interval-valued functions. The epigraph of a SLU$\mathcal{E}$P function is derived, and a relationship between SLU$\mathcal{E}$P and PSLU$\mathcal{E}$P functions have been explored. A key contribution of this work is the identification of a significant connection between weakly-strongly $\mathcal{E}$-invex functions and SLU$\mathcal{E}$P functions. As an application, we consider a nonlinear programming problem involving SLU$\mathcal{E}$P functions. Under certain conditions, we prove that a local minimum of the problem is also a global minimum. Moreover, the sufficiency of Karush-Kuhn-Tucker (KKT) optimality conditions by considering the objective and constraint functions are SLU$\mathcal{E}$I and S$\mathcal{E}$I respectively. The theoretical results are validated through illustrative examples and counterexamples.


\noindent
Tauheed\\  
{\textit{tauheedalicktd@gmail.com}}\\
\noindent
Akhlad Iqbal\\  
{\textit{akhlad.mm@amu.ac.in}}\\
\noindent
Amir Suhail\\  
{\textit{amir.mathamu@gmail.com}}\\
\noindent
Aligarh Muslim University, Aligarh-202002, India\\
\section{\textbf{introduction}}

  In recent years, generalized convexity has evolved considerably and remains fundamental in optimization theory, economics, engineering applications, and in the study of Riemannian manifolds, where it helps tackle diverse real-world problems.\cite{hanson,Akhlad1,Akhlad2,Akhlad3,Neogy}. Recognizing its broad applications, Youness \cite{Youness1999} introduced the concept of $\mathcal{E}$-convexity. Building on this, Youness \cite{Youness2013} extended the idea to strongly $\mathcal{E}$-convexity and explored its practical implications. Inspired by these developments, Majeed \cite{S.N.2} further enriched the theory by defining the strongly $\mathcal{E}$-convex hull, cone, and cone hull, and investigating their fundamental properties. Subsequently, Hussain and Iqbal \cite{Hussain} broadened the scope by introducing quasi strongly $\mathcal{E}$-convex functions and studied nonlinear programming problems within this framework.

 Another significant generalization is due to Fulga and Preda \cite{Fulga}, who defined $\mathcal{E}$-invex sets and $\mathcal{E}$-preinvex functions, providing a solid foundation for further advances in nonlinear optimization theory. Researchers like Wu have extended classical convex analysis to interval-valued functions, developing LU-convexity and proposing interval-valued KKT conditions \cite{Wu2007, Wu2009}.\\
Recently, Bhat et al.\ \cite{Bhat2024,Bhat2025a,Bhat2025b} have advanced this field by developing optimality conditions, establishing generalized Hukuhara directional differentiability, and deriving necessary conditions for interval-valued functions on Riemannian manifolds-highlighting the expanding role of interval analysis under generalized order relations.\\
Partial differentiability  for an IVF has been defined by Stefanini \cite{Stefanini2019} and Ghosh et al. \cite{Ghosh,Ghosh2019} but complete characterization of gH-partial differentiability has been discussed in \cite{Tauheed}.
Motivated by these influential works \cite{Bazaara,Lupsa,Pini,Suneja}, this paper introduces and investigates new classes of interval-valued functions under the LU ordering: strongly LU-$\mathcal{E}$-preinvex (SLU$\mathcal{E}$P), and pseudo strongly LU-$\mathcal{E}$-preinvex (PSLU$\mathcal{E}$P), strongly LU-$\mathcal{E}$-invex (SLU$\mathcal{E}$I) functions.\\
The definition of gH-gradient for an IVF and gH-product of a vector with n tuples of interals  in \cite{Tauheed} provides a lemma which helps us in defining strongly LU-$\mathcal{E}$-invex (SLU$\mathcal{E}$I) functions. These new generalizations extend the theory of strong invexity to the interval setting, capturing more intricate functional behaviors while preserving valuable analytical properties.
We present precise definitions of these classes and explore their main properties, supported by carefully constructed examples and counterexamples. Furthermore, we extend the concept of strongly G-invex (SGI) sets to the interval context, creating a unified framework for generalized convex sets and interval-valued functions.

A key contribution of this work is the characterization of the epigraph of SLU$\mathcal{E}$P functions and the establishment of clear relationships among SLU$\mathcal{E}$P, PSLU$\mathcal{E}$P, and weakly-strongly $\mathcal{E}$-invex functions, thereby enriching the understanding of their interconnections.

As an application, we formulate a nonlinear programming problem involving SLU$\mathcal{E}$P functions and demonstrate that, under suitable conditions, any local minimum is guaranteed to be a global minimum, thereby generalizing classical results. Moreover, the sufficiency of Karush-Kuhn-Tucker (KKT) optimality conditions by considering the objective and constraint functions are SLU$\mathcal{E}$I and S$\mathcal{E}$I respectively.

Overall, this paper adds to the growing body of research on optimization under uncertainty by expanding the theoretical framework of interval-valued functions and offering robust tools for tackling complex nonlinear problems. Theoretical results are reinforced with illustrative examples and counterexamples to clarify their scope and limitations.

  \section{\textbf{preliminaries}}
Let \( \mathcal{I}(\mathbb{R}) \) denote the set of all closed and bounded intervals on  \( \mathbb{R} \). For any interval \( \mathcal{U} \in \mathcal{I}(\mathbb{R}) \), it is defined as:

\[
\mathcal{U} = [u^L, u^U] \quad \text{where} \quad u^L, u^U \in \mathbb{R} \quad \text{and} \quad u^L \leq u^U.
\]

Given two intervals \( \mathcal{U}_1 = [u_1^L, u_1^U] \) and \( \mathcal{U}_2 = [u_2^L, u_2^U] \) in \( \mathcal{I}(\mathbb{R}) \), their sum is defined as:

\[
\mathcal{U}_1 + \mathcal{U}_2 = [u_1^L + u_2^L, u_1^U + u_2^U].
\]

The negation of \( \mathcal{U}_1 \) is given by:

\[
-\mathcal{U}_1 = [-u_1^U, -u_1^L].
\]

Consequently, the difference between \( \mathcal{U}_1 \) and \( \mathcal{U}_2 \) is expressed as:

\[
\mathcal{U}_1 - \mathcal{U}_2 = \mathcal{U}_1 + (-\mathcal{U}_2) = [u_1^L - u_2^U, u_1^U - u_2^L].
\]

Additionally, scalar multiplication of \( \mathcal{U}_1 \) by a real number \( k \) is defined as:

\[
k\mathcal{U}_1 = 
\begin{cases} 
[ku_1^L, ku_1^U] & \text{if } k \geq 0, \\
[ku_1^U, ku_1^L] & \text{if } k < 0.
\end{cases}
\]

This summarizes the fundamental operations on intervals within the set \( \mathcal{I}(\mathbb{R}) \).

For a deeper understanding of interval analysis, the interested reader is encouraged to consult the foundational works by Moore~\cite{Moore1979}, as well as the comprehensive treatment provided by Alefeld and Herzberger~\cite{AlefeldHerzberger}.

The Hausdorff distance between two intervals \( \mathcal{U }_1= [u_1^L, u_1^U] \) and \(\mathcal{ U}_2 = [u_2^L, u_2^U] \) is defined as:

\[
d_H(\mathcal{U}_1, \mathcal{U}_2) = \max\{|u_1^L - u_2^L|, |u_1^U - u_2^U|\}.
\]

A limitation of  standard interval subtraction is that, for any interval \(\mathcal{ U} \in \mathcal{ I(\mathbb{R})} \), the result of \( \mathcal{U - U }\) is not equal to zero. For instance, if \( \mathcal{U }= [0, 1] \), then:

\[
\mathcal{U - U} = [0, 1] - [0, 1] = [-1, 1] \neq 0.
\]

To resolve this issue, the \textit{Hukuhara difference} between two intervals \( \mathcal{U}_1 = [u_1^L, u_1^U] \) and \(\mathcal{ U }_2= [u_2^L, u_2^U] \) is introduced as:

$$\mathcal{U}_1 \ominus \mathcal{U}_2 = [u_1^L - u_2^L, u_1^U - u_2^U].$$

With this definition, for any interval \( \mathcal{U} \in \mathcal{ I(\mathbb{R})} \), \(\mathcal{ U \ominus U }= 0 \). However, the Hukuhara difference is not always valid for arbitrary intervals. For example, \( [0, 4] \ominus [0, 8] = [0, -4] \), which is not a interval since the lower bound exceeds the upper bound. This highlights a constraint in the applicability of the Hukuhara difference.

To overcome this limitation, Stefanini and Bede \cite{Stefanini2009} proposed the \textit{generalized Hukuhara difference} (gH-difference) for two intervals \( \mathcal{U}_1 \) and \(\mathcal{ U}_2 \), which is defined as:

\[
\mathcal{U}_1 \ominus_{gH} \mathcal{U}_2= \mathcal{U}_3 \iff 
\begin{cases} 
    (i) & \mathcal{U}_1 =\mathcal{ U}_2 + \mathcal{U}_3, \quad \text{or} \\ 
    (ii) & \mathcal{U}_2 = \mathcal{U}_1 - \mathcal{U}_3. 
\end{cases}
\]

In case \((i)\), the gH-difference is equivalent to the Hukuhara difference (H-difference) \cite{Wu2007}.

 For any two intervals \( \mathcal{U}_1 = [u_1^L, u_1^U] \) and \(\mathcal{ U}_2 = [u_2^L, u_2^U] \), the gH-difference \( \mathcal{U}_1 \ominus_{gH}\mathcal{ U}_2 \) always exists and is uniquely determined. Moreover, the following properties hold:
 \begin{equation*}
     \begin{aligned}
         \mathcal{U}_1 \ominus_{gH} \mathcal{U}_1 =& [0, 0]\\ \quad \text{and} \quad \\
\mathcal{U}_1 \ominus_{gH} \mathcal{U}_2 =& \left[\min\{u_1^L - u_2^L, u_1^U - u_2^U\}, \max\{u_1^L - u_2^L, u_1^U - u_2^U\}\right]
     \end{aligned}
 \end{equation*}

This generalized approach ensures that the difference between intervals is always well defined and resolves the issues associated with the standard Hukuhara difference.

 
\begin{definition} \cite{Wu2007}. \textbf{Order-relation} 
Let \( \mathcal{U}_1 = [u_1^L, u_1^U] \) and \( \mathcal{U}_2 = [u_2^L, u_2^U] \) be two closed intervals in \( \mathbb{R} \). We define the relation \( \ \preceq\) as follows:

\begin{equation*}
    \mathcal{U}_1 \preceq\mathcal{ U}_2 \quad \text{if and only if} \quad u_1^L \leq u_2^L \quad \text{and} \quad u_1^U \leq u_2^U.
\end{equation*}

Next, we define the strict inequality \(\mathcal{ U}_1 \prec\mathcal{ U}_2 \) as \(\mathcal{ U}_1 \preceq \mathcal{U}_2 \) and \( \mathcal{U}_1 \neq \mathcal{U}_2 \). This can be equivalently expressed as:

 \[
\begin{cases}
u_{1}^L < u_{2}^L \\
u_{1}^U \leq u_{2}^U
\end{cases}
\quad \text{or} \quad
\begin{cases}
u_{1}^L \leq u_{2}^L \\
u_{1}^U < u_{2}^U
\end{cases}
\quad \text{or} \quad
\begin{cases}
u_{1}^L < u_{2}^L \\
u_{1}^U < u_{2}^U.
\end{cases}
\]
\end{definition} 
\begin{proposition} \cite{Stefanini2009}
        \item[(i)] $\forall$ \( \mathcal{ U}_1, \mathcal{ U}_2 \in \mathcal{I}(\mathbb{R}) \), \( \mathcal{ U}_1 \ominus_{gH} \mathcal{ U}_2 \) always exists and \( \mathcal{ U}_1 \ominus_{gH} \mathcal{ U}_2 \in \mathcal{I}(\mathbb{R}) \).
        \item[(ii)] \( \mathcal{ U}_1 \ominus_{gH} \mathcal{ U}_2 \preceq 0 \) if and only if \( \mathcal{ U}_1 \preceq \mathcal{ U}_2 \).
\end{proposition}

\begin{definition}Let \( \mathcal{S} \subseteq \mathcal{I}(\mathbb{R}) \) be a finite subset. 
An interval \( \mathcal{U}_1 \in \mathcal{I}(\mathbb{R}) \) is said to be a \emph{maximum element} of \( \mathcal{S} \) if, for every \( \mathcal{U}_2 \in \mathcal{S} \), the following condition holds:

\[
\mathcal{U}_2 \preceq \mathcal{U}_1.
\]

In this case, we denote the maximum element as:

\[
\max \mathcal{S} = \mathcal{U}_1.
\]
\end{definition}


\begin{definition} \cite{Wu2007}.A function \( \tilde{h}: \mathbb{R}^n \to \mathcal{I}(\mathbb{R}) \) is said to be an interval-valued function (IVF) if, for each point \( \mathit{\zeta}  \in \mathbb{R}^n \), it can be expressed as:

\[
\tilde{h}(\mathit{\zeta} ) = \left[\tilde{h}^L(\mathit{\zeta} ), \tilde{h}^U(\mathit{\zeta} )\right],
\]

where \( \tilde{h}^L, \tilde{h}^U : \mathbb{R}^n \rightarrow \mathbb{R} \) are real-valued functions such that \( \tilde{h}^L(\mathit{\zeta} ) \leq \tilde{h}^U(\mathit{\zeta} ) \), $\forall$ \( \mathit{\zeta}  \in \mathbb{R}^n \).
\end{definition} 

\begin{definition} \cite{KumarandGhosh}
    Let $\mathcal{K} \subseteq \mathcal{I}(\mathbb{R})$. An interval $\bar{A} \in \mathcal{I}(\mathbb{R})$ is said to be a lower bound of $\mathcal{K}$ if 
\[
\bar{A} \preceq B \quad \text{for all } B \in \mathcal{K}.
\]
A lower bound $\bar{A}$ of $\mathcal{K}$ is called an infimum of $\mathcal{K}$ if for all lower bounds $C$ of $\mathcal{K}$ in $\mathcal{I}(\mathbb{R})$,
\[
C \preceq \bar{A}.
\]

Similarly, an interval $\bar{A} \in \mathcal{I}(\mathbb{R})$ is said to be an upper bound of $\mathcal{K}$ if 
\[
B \preceq \bar{A} \quad \text{for all } B \in \mathcal{K}.
\]
An upper bound $\bar{A}$ of $\mathcal{K}$ is called a supremum of $\mathcal{K}$ if for all upper bounds $C$ of $\mathcal{K}$ in $\mathcal{I}(\mathbb{R})$,
\[
\bar{A} \preceq C.
\]
\end{definition}
If lower bound of $\mathcal{K}$ does not exist then infimum of $\mathcal{K}$ is $+\infty$ ([$+\infty$,$+\infty$]).
If upper bound of $\mathcal{K}$ does not exist then supremum of $\mathcal{K}$ is $-\infty$ ([$-\infty$,$-\infty$]).
\begin{definition}\cite{KumarandGhosh}
Let $\mathcal{S}$ be a nonempty subset of $\mathbb{R}^n $ and 
$\mathsf{T} : \mathcal{S} \to {\mathcal{I}(\mathbb{R})}$
be an extended IVF. Then, the infimum of $\mathsf{T}$ denoted as 
$\inf_{\zeta \in \mathcal{S}} \mathsf{T}(\zeta)$
is equal to the infimum of the range set of $\mathsf{T}$, that is
\[
\inf_{\zeta \in \mathcal{S}} \mathsf{T}(\zeta)
    = \inf\{\mathsf{T}(\zeta) : \zeta \in \mathcal{S}\}.
\]

Similarly, the supremum of an IVF is defined by
\[
\sup_{\zeta \in \mathcal{S}} \mathsf{T}(\zeta)
    = \sup\{\mathsf{T}(\zeta) : \zeta \in \mathcal{S}\}.
\]
\end{definition}

 Building on this formal structure, Wu \cite{Wu2007} introduced a rigorous extension of classical calculus into the interval domain: the notions of continuity, limit, and two distinct forms of differentiability for interval-valued mappings. Next we give few Definitions and results that will be used in building our main results.


\begin{proposition} \cite{Wu2007} Suppose $\mathcal{S} \subseteq \mathbb{R}$ is an open set and $\tilde{h}: \mathcal{S} \rightarrow \mathcal{I}(\mathbb{R})$ is an IVF, with $\tilde{h}(\mathit{\zeta} ) = [\tilde{h}^L(\mathit{\zeta} ), \tilde{h}^U(\mathit{\zeta} )]$ if  $\tilde{h}^L$ and $\tilde{h}^U$ are differentiable at $\mathit{\zeta} _0$ in the classical sense. Then  $\tilde{h} $ is said to be weakly differentiable at a point $\mathit{\zeta} _0$
\end{proposition}
\vspace{0.2cm}


\begin{definition}\cite{Stefanini2009}
Let $\mathit{\zeta} _0 \in (a, b)$ and suppose $\tilde{h} $ is such that $\mathit{\zeta} _0 + t \in (a, b)$. The gH-derivative of an IVF $\tilde{h} : (a, b) \to \mathcal{I}(\mathbb{R})$ at $\mathit{\zeta} _0$ is given as follows:
\begin{equation*}
    \tilde{h}'(\mathit{\zeta} _0) = \lim_{t \to 0} \frac{\tilde{h}(\mathit{\zeta} _0 + t) \ominus_{gH}\tilde{h}(\mathit{\zeta} _0)}{t}.
\end{equation*}
If this limit exists and lies in $\mathcal{I}(\mathbb{R})$, then $\tilde{h} $ is said to be gH-differentiable at $\mathit{\zeta} _0$.
\end{definition}
 
Let $\mathbb{R}^n$ denote the $n$-dimensional Euclidean space, and let $\mathcal{E}: \mathbb{R}^n \rightarrow \mathbb{R}^n$ be a mapping defined on this space. To extend classical convexity, Youness~\cite{Youness2013} introduced the concept of strongly $\mathcal{E}$-convex functions, which generalize the notion of convexity by involving the map $\mathcal{E}$ to capture broader structural properties of functions and sets.
 
\begin{definition}
	\cite{Youness2013}\label{definition 2.5 p4}
	A non-empty set $\mathcal{S} \subseteq \mathbb{R}^n$ is said to be strongly $\mathcal{E}$-convex with respect to a mapping $\mathcal{E}: \mathbb{R}^n \rightarrow \mathbb{R}^n$ if, for any $\zeta, \delta \in \mathcal{S}$ and for all $\alpha, \lambda \in [0,1]$, the following condition holds:
\[
\lambda \big( \alpha \zeta + \mathcal{E}(\zeta) \big) + (1-\lambda) \big( \alpha \delta + \mathcal{E}(\delta) \big) \in \mathcal{S}.
\]

\end{definition}
\begin{definition}\label{defintion 2.6 p4}
	\cite{Youness2013}
	Let $\mathcal{S} \subseteq \mathbb{R}^n$ be a strongly $\mathcal{E}$-convex set with respect to a mapping $\mathcal{E}: \mathbb{R}^n \rightarrow \mathbb{R}^n$. A  function $\tilde{h}: \mathcal{S} \subseteq \mathbb{R}^n \to \mathbb{R}$ is said to be strongly $\mathcal{E}$-convex on $\mathcal{S}$ if, for every $\zeta, \delta \in \mathcal{S}$ and all $\alpha \in [0,1]$ and $\lambda \in [0,1]$, the following condition holds:
\[
\tilde{h}\Big(\lambda (\alpha \zeta + \mathcal{E}(\zeta)) + (1-\lambda)(\alpha \delta + \mathcal{E}(\delta))\Big)
\leq \lambda \tilde{h}\big(\mathcal{E}(\zeta)\big) + (1-\lambda)\tilde{h}\big(\mathcal{E}(\delta)\big).
\]

If the inequality is strict whenever $\alpha \zeta + \mathcal{E}(\zeta) \neq \alpha \delta + \mathcal{E}(\delta)$ and $\lambda \in (0,1)$, then $\tilde{h}$ is called strictly strongly $\mathcal{E}$-convex.

Moreover, when $\alpha = 0$, the notion of strongly $\mathcal{E}$-convexity reduces to the standard $\mathcal{E}$-convexity.
\end{definition}

Fulga et al.\ \cite{Fulga} extended the concept of $\mathcal{E}$-convexity originally introduced by Youness \cite{Youness2001a} to a broader notion called $\mathcal{E}$-preinvexity, which is defined as follows:

  \begin{definition} \cite{Fulga} \label{definition 2.7 p4}
 	A non empty subset $ {\mathcal{S}}\subseteq  {\mathbb{R}^n}$ is called    $ {\mathcal{E}}$-invex set  w.r.t. $ {\Psi}: {\mathbb{R}^n}\times  {\mathbb{R}^n}\rightarrow {\mathbb{R}^n}$, if for every $\mathit{\zeta} ,\mathit{\delta} \in {\mathcal{S}}$ and $ {\lambda}\in  {[0,1]} $, 
 	$${\mathcal{E}} (\mathit{\delta} )+{\lambda}{\Psi}({\mathcal{E}}(\mathit{\zeta} ),{\mathcal{E}} (\mathit{\delta} ))\in {\mathcal{S}}.$$ 
 \end{definition}
 \begin{definition} \cite{Fulga}\label{definition 2.8 p4}
 	 \noindent
Let $\mathcal{S}$ be an $\mathcal{E}$-invex set. A function $\tilde{h} : \mathcal{S} \rightarrow \mathbb{R}$ is said to be $\mathcal{E}$-preinvex with respect to $\Psi$ on $\mathcal{S}$ if, for all $\zeta, \delta \in \mathcal{S}$ and for every $\lambda \in [0,1]$, the following inequality holds:
\[
\tilde{h}\big(\mathcal{E}(\delta) + \lambda\, \Psi(\mathcal{E}(\zeta), \mathcal{E}(\delta))\big)
\leq \lambda\, \tilde{h}\big(\mathcal{E}(\zeta)\big) + (1 - \lambda)\, \tilde{h}\big(\mathcal{E}(\delta)\big).
\]

 \end{definition}

 \begin{definition}\label{definition 2.9 p4} \cite{Iqbal2022}
  	\noindent
A set $\mathcal{S} \subseteq \mathbb{R}^n$ is said to be \emph{strongly} $\mathcal{E}$-\emph{invex} with respect to a map $\Psi : \mathbb{R}^n \times \mathbb{R}^n \rightarrow \mathbb{R}^n$ if, for all $\zeta, \delta \in \mathcal{S}$ and for every $\alpha, \lambda \in [0,1]$, the following condition holds:
\[
\alpha\, \delta + \mathcal{E}(\delta) + \lambda\, \Psi\big(\alpha\, \zeta + \mathcal{E}(\zeta),\, \alpha\, \delta + \mathcal{E}(\delta)\big) \in \mathcal{S}.
\]
If $\Psi\big(\alpha\, \zeta + \mathcal{E}(\zeta),\, \alpha\, \delta + \mathcal{E}(\delta)\big) = \big(\alpha\, \zeta + \mathcal{E}(\zeta)\big) - \big(\alpha\, \delta + \mathcal{E}(\delta)\big)$, then $\mathcal{S}$ coincides with the notion of a strongly $\mathcal{E}$-convex set (see Definition~\ref{definition 2.5 p4}). Furthermore, setting $\alpha = 0$ reduces $\mathcal{S}$ to an $\mathcal{E}$-invex set (see Definition~\ref{definition 2.7 p4}), when $\alpha = 0$ and $\Psi\big(\mathcal{E}(\zeta),\, \mathcal{E}(\delta)\big) = \mathcal{E}(\zeta) - \mathcal{E}(\delta)$, the concept reduces to the classical $\mathcal{E}$-convex set introduced by Youness~\cite{Youness2001a}.

  \end{definition}

 \begin{lemma}\label{lemma 2.1 p4} \cite{Iqbal2022}
  	\noindent
Whenever a set $\mathcal{S} \subseteq \mathbb{R}^n$ is strongly $\mathcal{E}$-invex relative to the mapping $\Psi : \mathbb{R}^n \times \mathbb{R}^n \to \mathbb{R}^n$, the image of $\mathcal{S}$ under $\mathcal{E}$ is contained within $\mathcal{S}$; that is, $\mathcal{E}(\mathcal{S}) \subseteq \mathcal{S}$.

  \end{lemma}
  \begin{lemma}\label{2}  \cite{Iqbal2022}
  	\noindent
Consider a collection $\{\mathcal{S}_i\}_{i \in I}$ of nonempty subsets of $\mathbb{R}^n$, each of which is strongly $\mathcal{E}$-invex with respect to a mapping $\Psi : \mathbb{R}^n \times \mathbb{R}^n \rightarrow \mathbb{R}^n$. Then the intersection of all these sets, denoted by $\mathcal{S} = \bigcap_{i \in I} \mathcal{S}_i$, remains strongly $\mathcal{E}$-invex with respect to the same mapping $\Psi$. 
  \end{lemma}
 
\begin{definition}\cite{Tauheed}
    Let \( \nu = (\nu_1, \nu_2, \ldots, \nu_n) \in \mathbb{R}^n \) and 
\( \tilde{\mathcal{U}} = (\mathcal{U}_1, \mathcal{U}_2, \ldots, \mathcal{U}_n) \in \mathcal{I}^n(\mathbb{R}) \).  
Define the index sets
\[
j^+ = \{\, i \mid \nu_i \geq 0 \,\}, \qquad 
j^- = \{\, i \mid \nu_i < 0 \,\}.
\]
Amir et al.~\cite{Tauheed} introduce the $gH$-product as follows,
\[
\langle \cdot , \cdot \rangle_{\textbf{gH}} : \mathbb{R}^n \times \mathcal{I}^n(\mathbb{R}) \longrightarrow \mathcal{I}(\mathbb{R})
\]
given by
\begin{equation*}
\langle \nu , \tilde{\mathcal{U}} \rangle_{\textbf{gH}}
    = \sum_{i \in j^+} \nu_i \mathcal{U}_i 
    \;\ominus_{gH}\;
      \sum_{k \in j^-} |\nu_k| \mathcal{U}_k .
\end{equation*}
is called gH-product of a vector $\nu$ with $\tilde{\mathcal{U}}$

\end{definition}
  
  \section{ \textbf{main results }}
  \justifying
  In this section, we introduce and study three new classes of interval-valued functions: strongly LU $\mathcal{\mathcal{E}}$-preinvex(SLU$\mathcal{E}$P), semi-strongly LU $\mathcal{\mathcal{E}}$-preinvex  and strongly LU $\mathcal{\mathcal{E}}$-invex (SLU$\mathcal{E}$I) functions. To demonstrate the existence and applicability of these concepts, we construct illustrative examples. Furthermore, we explore several fundamental properties and results associated with these functions.

\noindent
Motivated by the work of Iqbal et al. \cite{Iqbal2022}, who examined nonlinear programming problems involving strongly $\mathcal{\mathcal{E}}$-invex sets and strongly $\mathcal{\mathcal{E}}$-preinvex functions, we now define the notion of a SLU$\mathcal{E}$P function.

      \begin{definition} \label{definition 3.1 p3}
      	Let ${\mathcal{S}}\subseteq { \mathbb{R}^{n}}$ be a non empty S$\mathcal{E}I$ set. A function $ \tilde{h}: {\mathcal{S}}\subseteq  \mathbb{R}^{n}\rightarrow  \mathcal{I}(\mathbb{R})$ is said to be SLU$\mathcal{E}$P w.r.t. ${\Psi}$ on ${\mathcal{S}}$, if $\forall~\mathit{\zeta} ,\mathit{\delta} \in  {\mathcal{S}},~{\alpha}\in {[0,1]}~\&~{\lambda} \in {[0,1]} $, we have
      	$$ \tilde{h}\big({\alpha}  \mathit{\delta} + {\mathcal{E}}( \mathit{\delta} )+ {\lambda} {\Psi}( {\alpha}  \mathit{\zeta} + {\mathcal{E}}( \mathit{\zeta} ), {\alpha}  \mathit{\delta} + {\mathcal{E}}( \mathit{\delta} ))\big)\preceq{\lambda}\tilde{h}\big({\mathcal{E}}( \mathit{\zeta} )\big)+(1- {\lambda})\tilde{h}\big({\mathcal{E}} (\mathit{\delta} )\big).$$ 
        In particular, when $\alpha=0$, then $\tilde{h}$ is LU-E-preinvex. Also, if $\mathcal E{(x)}=x$ and $\alpha=0$ then $\tilde{h}$ is LU-preinvex.
      	If the above inequality is strict for all $\mathit{\zeta} ,\mathit{\delta} \in  {\mathcal{S}}$,~$ {\alpha} \mathit{\zeta} + {\mathcal{E}}(\mathit{\zeta} )\neq  {\alpha} \mathit{\delta} + {\mathcal{E}} (\mathit{\delta} ),~\forall~  {\alpha}\in {[0,1]}~\&~{\lambda}\in {(0,1)}$. Then, the function $\tilde{h}$ is called strictly SLU$\mathcal{E}$P.\\

      		\noindent
      For the real valued functions, if $\Psi\big(\alpha \mathit{\zeta} +\mathcal{E}(\mathit{\zeta} ),\alpha \mathit{\delta} +\mathcal{E} (\mathit{\delta} )\big)=\big(\alpha \mathit{\zeta} +\mathcal{E}(\mathit{\zeta} )-(\alpha \mathit{\delta} + \mathcal{E}(\mathit{\delta} ))\big)$, then the function $\tilde{h}$ reduces to a strongly-$\mathcal{E}$-convex (S$\mathcal{E}$C), defined in Definition \ref{defintion 2.6 p4}. 
      \end{definition}

\begin{definition}
   Let $\mathcal{S} \subseteq \mathbb{R}^{n}$ be a nonempty S$\mathcal{\mathcal{E}}$I set.  
A function $\tilde{h} \colon \mathbb{R}^{n} \to \mathcal{I}(\mathbb{R})$ is called semi strongly LU-$\mathcal{\mathcal{E}}$-preinvex (SSLU$\mathcal{\mathcal{E}}$P) w.r.t. $\Psi$ on $\mathcal{S}$, if  
\begin{equation*}
    \tilde{h}(\alpha \mathit{\delta}  + \mathcal{\mathcal{E}}(\mathit{\delta} ) + \lambda \Psi(\alpha \mathit{\zeta}  + \mathcal{\mathcal{E}}(\mathit{\zeta} ), \alpha \mathit{\delta}  + \mathcal{\mathcal{E}}(\mathit{\delta} ))) 
    \preceq \lambda \tilde{h}(\mathit{\zeta} ) + (1 - \lambda) \tilde{h}(\mathit{\delta} ),
\end{equation*}
$\forall$  $\mathit{\zeta} , \mathit{\delta}  \in \mathcal{S}$ and $\alpha, \lambda \in [0,1]$. 
\end{definition}
In particular, when $\alpha=0$, then $\tilde{h}$ is semi LU-E-preinvex. Also, if $\mathcal E{(x)}=x$ and $\alpha=0$ then $\tilde{h}$ is semi LU-preinvex.\\
We now present an example of a function that satisfies the conditions of a SLUEP function but does not qualify as a  SSLU$\mathcal{E}$P function.

      \begin{example} \label{example3.1}
Let $\tilde{h}:\mathbb{R} \rightarrow{\mathcal{I }(\mathbb{R})}$ be an IVF defined as follows:
\[
\tilde{h}(\mathit{\zeta} ) =
\begin{cases} 
    [0,1] & \text{if } \mathit{\zeta}  > 0 \\
     [\mathit{\zeta} , \mathit{\zeta}  + 1] & \text{if } \mathit{\zeta}  \leq 0
\end{cases}
\]

and let \( \mathcal{\mathcal{E}}: \mathbb{R} \to \mathbb{R} \) be a map defined as

\[
\mathcal{\mathcal{E}}(\mathit{\zeta} ) = |\mathit{\zeta} |
\]

Also, let \( \Psi : \mathbb{R} \times \mathbb{R} \to \mathbb{R} \) be defined as

\[
\Psi(\mathit{\zeta} , \mathit{\delta} ) =
\begin{cases}
    \mathit{\zeta}  + \mathit{\delta} , & \text{if } \mathit{\zeta} , \mathit{\delta}  \geq 0 ~\text{ or }~ \mathit{\zeta} , \mathit{\delta}  \leq 0 \\
    -\mathit{\delta} , & \text{otherwise}
\end{cases}
\]

IVF \( \tilde{h} \) is SLU\(\mathcal{\mathcal{E}}\)P  but it is not SSLU\(\mathcal{\mathcal{E}}\)P.\\
\textbf{Case I}: If \( \mathit{\zeta}  > 0 \) and \( \mathit{\delta}  < 0 \), then
\begin{align*}
    \tilde{h} \left( \alpha \mathit{\delta}  + \mathcal{\mathcal{E}}(\mathit{\delta} ) + \lambda\Psi(\alpha \mathit{\zeta}  + \mathcal{\mathcal{E}}(\mathit{\zeta} ), \alpha \mathit{\delta}  + \mathcal{\mathcal{E}}(\mathit{\delta} )) \right) &\preceq \lambda \tilde{h}(\mathcal{E}(\mathit{\zeta} )) + (1 - \lambda) \tilde{h}(\mathcal{E}(\mathit{\delta} )), \\
    \tilde{h} (\alpha\mathit{\delta}  - \mathit{\delta}  + \lambda\Psi(\alpha\mathit{\zeta} +\mathit{\zeta} , \alpha\mathit{\delta} -\mathit{\delta} )) &\preceq \lambda \tilde{h}(\mathit{\zeta} ) + (1 - \lambda) \tilde{h}(-\mathit{\delta} ), \\
    \tilde{h} (\alpha\mathit{\delta}  - \mathit{\delta}  + \lambda(\alpha\mathit{\zeta} +\mathit{\zeta} +\alpha\mathit{\delta}  - \mathit{\delta}  ) &\preceq \lambda[0,1]+(1-\lambda)[0,1] \} \\ 
    \tilde{h} (-\mathit{\delta} (1-\alpha)+\lambda((1+\alpha)\mathit{\zeta} -\mathit{\delta}  (1-\alpha)  &\preceq [0,1] \\
    [0,1] &\preceq [0,1]\\
\end{align*}
\noindent
\textbf{Case II}: If \( \mathit{\zeta}  < 0 \) and \( \mathit{\delta}  > 0 \), then
\begin{align*}
    \tilde{h} \left( \alpha \mathit{\delta}  + \mathcal{\mathcal{E}}(\mathit{\delta} ) + \lambda\Psi(\alpha \mathit{\zeta}  + \mathcal{\mathcal{E}}(\mathit{\zeta} ), \alpha \mathit{\delta}  + \mathcal{\mathcal{E}}(\mathit{\delta} )) \right) &\preceq \lambda \tilde{h}(\mathcal{E}(\mathit{\zeta} )) + (1 - \lambda) \tilde{h}(\mathcal{E}(\mathit{\delta} )), \\
    \tilde{h} (\alpha\mathit{\delta}  + \mathit{\delta}  + \lambda\Psi(\alpha\mathit{\zeta} -\mathit{\zeta} , \alpha\mathit{\delta} +\mathit{\delta} )) &\preceq \lambda \tilde{h}(-\mathit{\zeta} ) + (1 - \lambda) \tilde{h}(\mathit{\delta} ), \\
    \tilde{h} (\alpha\mathit{\delta}  + \mathit{\delta}  + \lambda(\alpha\mathit{\zeta} -\mathit{\zeta} +\alpha\mathit{\delta}  + \mathit{\delta}  ) &\preceq \lambda[0,1]+(1-\lambda)[0,1] \} \\ 
    \tilde{h} (\mathit{\delta} (1+\alpha)+\lambda(-(1-\alpha)\mathit{\zeta} +\mathit{\delta}  (1+\alpha)  &\preceq [0,1] \\
    [0,1] &\preceq [0,1]\\
\end{align*}
\noindent
Similarly we can check for the rest of the cases.\\
Therefore, it is SLU\(\mathcal{\mathcal{E}}\)P. Moreover it is not SSLU\(\mathcal{\mathcal{E}}\)P.\\
Particularly at the points, $\mathit{\zeta} =0,\mathit{\delta} =-1,\alpha=\frac{1}{2}$ and $\lambda=0$, we get 
\begin{equation*}
    \begin{aligned}
         \tilde{h} \left( \alpha \mathit{\delta}  + \mathcal{\mathcal{E}}(\mathit{\delta} ) + \lambda\Psi(\alpha \mathit{\zeta}  + \mathcal{\mathcal{E}}(\mathit{\zeta} ), \alpha \mathit{\delta}  + \mathcal{\mathcal{E}}(\mathit{\delta} )) \right)&=\tilde{h}(-\frac{1}{2}+1+0)\\
         &=\tilde{h}(\frac{1}{2})\\
         &=[0,1]\\
    \end{aligned}
\end{equation*}
\begin{equation*}
    \begin{aligned}
        \lambda \tilde{h}(\mathit{\zeta} ) + (1 - \lambda) \tilde{h}(\mathit{\delta} )&=0+\tilde{h}(-1)\\
        &=[-1,0]
    \end{aligned}
\end{equation*}
\begin{equation*}
    \tilde{h}(\alpha \mathit{\delta}  + \mathcal{\mathcal{E}}(\mathit{\delta} ) + \lambda \Psi(\alpha \mathit{\zeta}  + \mathcal{\mathcal{E}}(\mathit{\zeta} ), \alpha \mathit{\delta}  + \mathcal{\mathcal{E}}(\mathit{\delta} ))) 
    \npreceq \lambda \tilde{h}(\mathit{\zeta} ) + (1 - \lambda) \tilde{h}(\mathit{\delta} ),
\end{equation*}

\begin{figure}[H]
		\centering
		\includegraphics[width=0.9\linewidth]{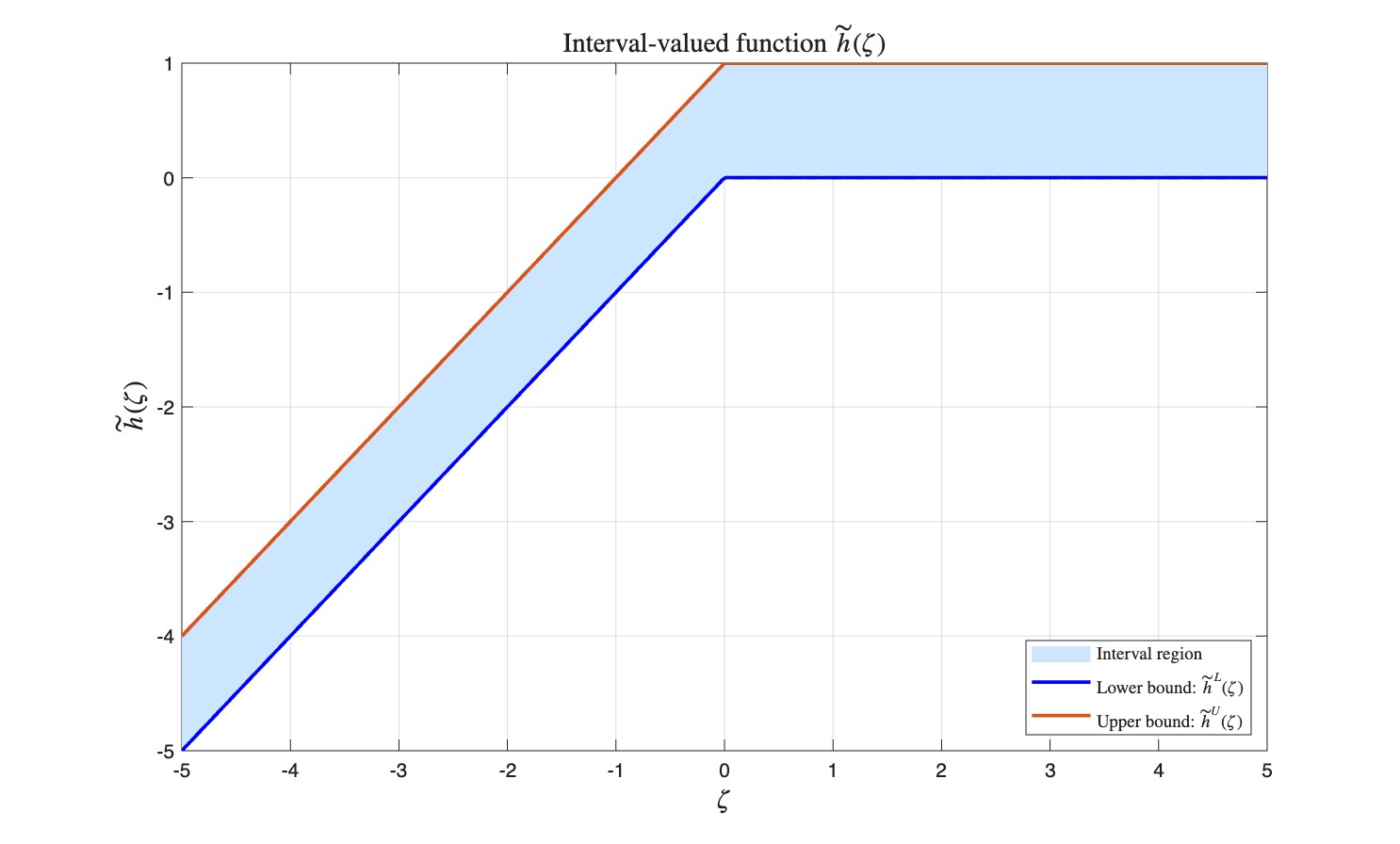}
		\label{fig:placeholder}
		\caption{ $\tilde{h}$ is SLU$\mathcal{E}$P but not SSLU\(\mathcal{\mathcal{E}}\)P with respect to the mappings $\mathcal{E}$ and $\Psi$ given in Example~\ref{example3.1}.
}
\end{figure}
\end{example}

The example below shows that an SSLU$\mathcal{E}$P function does not necessarily imply  SLU$\mathcal{E}$P function.

\begin{example} \label{example3.2}
   Let $\tilde{h} : \mathbb{R} \to \mathcal{I}(\mathbb{R})$ be an IVF defined by
\[
\tilde{h}(\zeta) = [\zeta,\, \zeta + 1].
\]
Let $\mathcal{E} : \mathbb{R} \to \mathbb{R}$ be a mapping defined by $\mathcal{E}(\zeta) = -\zeta$
\

Also, define $\Psi : \mathbb{R} \times \mathbb{R} \to \mathbb{R}$ as $\Psi(\zeta, \delta) = \zeta - \delta.$

IVF \( \tilde{h} \) is SSLU\(\mathcal{\mathcal{E}}\)P  but it is not SLU\(\mathcal{\mathcal{E}}\)P.\\

To show that IVF \( \tilde{h} \) is SSLU\(\mathcal{\mathcal{E}}\)P, see the following steps: 
\begin{equation*}
    \begin{aligned}
        \tilde{h}(\alpha \mathit{\delta}  + \mathcal{\mathcal{E}}(\mathit{\delta} ) + \lambda \Psi(\alpha \mathit{\zeta}  + \mathcal{\mathcal{E}}(\mathit{\zeta} ), \alpha \mathit{\delta}  + \mathcal{\mathcal{E}}(\mathit{\delta} ))) 
    &\preceq \lambda \tilde{h}(\mathit{\zeta} ) + (1 - \lambda) \tilde{h}(\mathit{\delta} )\\
    \tilde{h}(\alpha \mathit{\delta} -\mathit{\delta}  + \lambda \Psi(\alpha \mathit{\zeta}  -\mathit{\zeta} , \alpha \mathit{\delta} -\mathit{\delta} ) 
    &\preceq \lambda [\mathit{\zeta} ,\mathit{\zeta} +1] + (1 - \lambda)[\mathit{\delta} ,\mathit{\delta} +1]\\
    \tilde{h}(\alpha \mathit{\delta} -\mathit{\delta}  + \lambda(\alpha \mathit{\zeta}  -\mathit{\zeta} -\alpha \mathit{\delta} +\mathit{\delta} ) 
    &\preceq  [\lambda\mathit{\zeta} +(1 - \lambda)\mathit{\delta} ,\lambda\mathit{\zeta} +(1 - \lambda)\mathit{\delta} +1]\\
    \tilde{h}((\alpha-1)( \lambda\mathit{\zeta} +(1 - \lambda)\mathit{\delta} ))
    &\preceq  [\lambda\mathit{\zeta} +(1 - \lambda)\mathit{\delta} ,\lambda\mathit{\zeta} +(1 - \lambda)\mathit{\delta} +1]\\
     [(\alpha-1)(\lambda\mathit{\zeta} +(1 - \lambda)\mathit{\delta} ),(\alpha-1)(\lambda\mathit{\zeta} +(1 - \lambda)\mathit{\delta} )+1]
     &\preceq  [\lambda\mathit{\zeta} +(1 - \lambda)\mathit{\delta} ,\lambda\mathit{\zeta} +(1 - \lambda)\mathit{\delta} +1]
    \end{aligned}
\end{equation*}
Therfore $\tilde{h}$ is SSLU$\mathcal{\mathcal{E}}$P, but it is not SLU$\mathcal{\mathcal{E}}$P.\\
Particularly at the points,  $\mathit{\zeta} =0,\mathit{\delta} =1,\alpha=\frac{1}{2}$ and $\lambda=\frac{1}{2}$, we get 
\begin{equation*}
    \begin{aligned}
         \tilde{h}(\alpha \mathit{\delta}  + \mathcal{\mathcal{E}}(\mathit{\delta} ) + \lambda \Psi(\alpha \mathit{\zeta}  + \mathcal{\mathcal{E}}(\mathit{\zeta} ), \alpha \mathit{\delta}  + \mathcal{\mathcal{E}}(\mathit{\delta} )))
         &= \tilde{h}(\frac{1}{2}-1 + \frac{1}{2}\Psi(0, \frac{1}{2}-1))\\
         &=\tilde{h}(-\frac{1}{2} + \frac{1}{2}\Psi(0, -\frac{1}{2}))\\
         &=\tilde{h}(-\frac{1}{2} + \frac{1}{4})\\
         &=\tilde{h}(- \frac{1}{4})\\
         &=[- \frac{1}{4}, \frac{3}{4}]
    \end{aligned}
\end{equation*}
\begin{equation*}
    \begin{aligned}
        \lambda \tilde{h}(\mathcal{E}(\mathit{\zeta} )) + (1 - \lambda) \tilde{h}(\mathcal{E}(\mathit{\delta} ))
        &=\frac{1}{2}\tilde{h}(0)+\frac{1}{2}\tilde{h}(-1)\\
        &=\frac{1}{2}[0,1]+\frac{1}{2}[-1,0]\\
        &=[0,\frac{1}{2}]+[-\frac{1}{2},0]\\
        &=[-\frac{1}{2},\frac{1}{2}]
    \end{aligned}
\end{equation*}
 \begin{equation*}
     \tilde{h} \left( \alpha \mathit{\delta}  + \mathcal{\mathcal{E}}(\mathit{\delta} ) + \lambda\Psi(\alpha \mathit{\zeta}  + \mathcal{\mathcal{E}}(\mathit{\zeta} ), \alpha \mathit{\delta}  + \mathcal{\mathcal{E}}(\mathit{\delta} )) \right) \npreceq \lambda \tilde{h}(\mathcal{E}(\mathit{\zeta} )) + (1 - \lambda) \tilde{h}(\mathcal{E}(\mathit{\delta} ))
 \end{equation*}

\begin{figure}[H]
		\centering
		\includegraphics[width=0.8\linewidth]{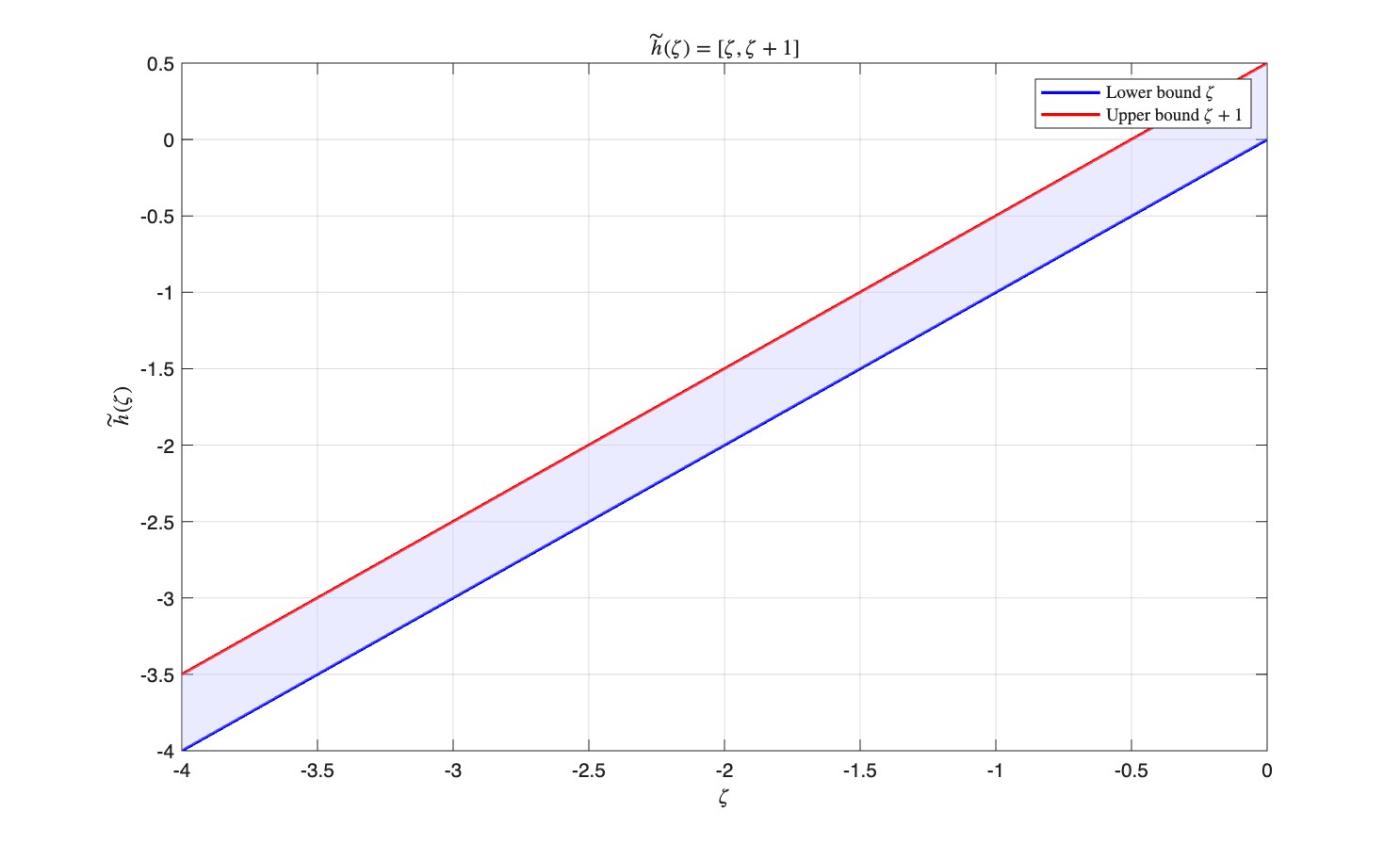}
		\label{fig:placeholder}
		\caption{ $\tilde{h}$ is SSLU$\mathcal{E}$P but not SLU$\mathcal{E}$P with respect to the mappings $\mathcal{E}$ and $\Psi$ given in Example~\ref{example3.2}.
}
\end{figure}

\end{example}
 
\begin{theorem} \label{theorem 3.1 p4}
    Suppose $\mathcal{S}$ is a S$\mathcal{\mathcal{E}}$I subset of $\mathbb{R}^n$ with respect to $\Psi : \mathcal{S} \times \mathcal{S} \rightarrow \mathbb{R}^n$, and let $\tilde{h}$ be an IVF. Then $\tilde{h}$ is SLU$\mathcal{\mathcal{E}}$P iff both endpoint functions $\tilde{h}^L$ and $\tilde{h}^U$ are S$\mathcal{\mathcal{E}}$P with respect to $\Psi$, i.e.,
\begin{equation}
   \begin{aligned}
       \tilde{h}^L\left(\alpha \mathit{\delta}  + \mathcal{\mathcal{E}}(\mathit{\delta} ) + \lambda \Psi(\alpha \mathit{\zeta}  + \mathcal{\mathcal{E}}(\mathit{\zeta} ), \alpha \mathit{\delta}  + \mathcal{\mathcal{E}}(\mathit{\delta} )))\right) 
       \leq \lambda \tilde{h}^L(\mathcal{\mathcal{E}}(\mathit{\zeta} )) + (1 - \lambda) \tilde{h}^L(\mathcal{\mathcal{E}}(\mathit{\delta} )),
   \end{aligned}
\end{equation}

\begin{equation}
    \begin{aligned}
        \tilde{h}^U\left(\alpha \mathit{\delta}  + \mathcal{\mathcal{E}}(\mathit{\delta} ) + \lambda \Psi(\alpha \mathit{\zeta}  + \mathcal{\mathcal{E}}(\mathit{\zeta} ), \alpha \mathit{\delta}  + \mathcal{\mathcal{E}}(\mathit{\delta} )))\right) 
        \leq \lambda \tilde{h}^U(\mathcal{\mathcal{E}}(\mathit{\zeta} )) + (1 - \lambda) \tilde{h}^U(\mathcal{\mathcal{E}}(\mathit{\delta} ))
    \end{aligned}
\end{equation}
for every $\lambda \in [0,1]$ and $\mathit{\zeta} , \mathit{\delta}  \in \mathcal{S}$.
\begin{proof}
     The conclusion is an immediate consequence of Definition~\ref{definition 3.1 p3}.
\end{proof}
\end{theorem}

     \begin{theorem}
    		If $\tilde{h}:{\mathcal{S}}\subseteq \mathbb{R}^{n}\rightarrow \mathcal{I}(\mathbb{R})$ is SLU$\mathcal{E}$P function on a S$\mathcal{E}I$ set ${\mathcal{S}}$, then $\tilde{h}\big({\alpha}\mathit{\delta} + {\mathcal{E}} (\mathit{\delta} )\big)\preceq\tilde{h}\big({\mathcal{E}} (\mathit{\delta} )\big)$, $\forall~\mathit{\delta} \in {\mathcal{S}}~\&~ {\alpha}\in {[0,1]}$.  
    \end{theorem}
    \begin{proof}
    	
    	If ${\lambda}=0$, then we get the desired result.   
    \end{proof}
    \begin{lemma}\label{lemma 3.1 p4}[Preservation of LU Order under Non-Negative Linear Combinations]
Let $\nu_1, \nu_2 \geq 0$ and let $\mathcal{I}_1, \mathcal{I}_2, \mathcal{I}_3, \mathcal{I}_4 \in \mathcal{I}(\mathbb{R})$ be intervals such that
\[
\mathcal{I}_1 \preceq \mathcal{I}_2 \quad \text{and} \quad \mathcal{I}_3 \preceq \mathcal{I}_4.
\]
Then,
\[
\nu_1 \mathcal{I}_1 + \nu_2 \mathcal{I}_3 \preceq \nu_1 \mathcal{I}_2 + \nu_2 \mathcal{I}_4.
\]
\end{lemma}

\begin{proof}
 The result follows directly from the definitions.
\end{proof}

    \begin{theorem}
    	Let $ {\mathcal{S}}\subseteq  \mathbb{R}^{n}$ be a S$\mathcal{E}I$ set. Suppose $ \tilde{h}_j: {\mathcal{S}}\subseteq  \mathbb{R}^{n}\rightarrow  \mathcal{I}(\mathbb{R})$,~$1\leq j\leq m$, are  SLU$\mathcal{E}$P  on $ {\mathcal{S}}$, then linear combination of  these functions is also preserves the SLU$\mathcal{E}$P property, i.e. for $\nu_j\geq0,~1\leq j\leq m $, then the function $${g}=\sum\limits_{j=1}^{m}\nu_j  \tilde{h}_j$$ is a SLU$\mathcal{E}$P on ${\mathcal{S}}$.  
    \end{theorem}
    
    \begin{proof}
    	Since $ \tilde{h}_j,~1\leq j\leq m$, are SLU$\mathcal{E}$P functions on S$\mathcal{E}I$ set ${\mathcal{S}}$, then   $\forall~\mathit{\zeta} ,\mathit{\delta} \in  {\mathcal{S}}$ and $ {\alpha},{\lambda}\in {[0,1]} $, we have\\
    	\begin{equation*}
    		 {\alpha} \mathit{\delta} + {\mathcal{E}} (\mathit{\delta} )+ {\lambda} {\Psi}\big( {\alpha} \mathit{\zeta} + {\mathcal{E}}(\mathit{\zeta} ), {\alpha} \mathit{\delta} + {\mathcal{E}} (\mathit{\delta} )\big)\in{\mathcal{S}},
    	\end{equation*}
        and
        \begin{equation*}
             \tilde{h}_j \left( \alpha \mathit{\delta}  + \mathcal{\mathcal{E}}(\mathit{\delta} ) + \lambda\Psi(\alpha \mathit{\zeta}  + \mathcal{\mathcal{E}}(\mathit{\zeta} ), \alpha \mathit{\delta}  + \mathcal{\mathcal{E}}(\mathit{\delta} )) \right) \preceq \lambda \tilde{h}_j(\mathcal{E}(\mathit{\zeta} )) + (1 - \lambda) \tilde{h}_j(\mathcal{E}(\mathit{\delta} )).
        \end{equation*}
   Multiplication by $\nu_j \geq 0$ preserves the LU order. By the Lemma \ref{lemma 3.1 p4} on non-negative linear combinations, summing yields:\\
    ${g}\big({\alpha}\mathit{\delta} + {\mathcal{E}} (\mathit{\delta} )+ {\lambda} {\Psi}( {\alpha}\mathit{\zeta} +{\mathcal{E}}(\mathit{\zeta} ), {\alpha}\mathit{\delta} + {\mathcal{E}} (\mathit{\delta} ))\big)$
    	\begin{equation*}
    	    \begin{aligned}
    	        \hspace{3cm} &=\sum\limits_{j=1}^{m}\nu_j\tilde{h}_j\big( {\alpha}\mathit{\delta} + {\mathcal{E}} (\mathit{\delta} )+ {\lambda} {\Psi}( {\alpha}\mathit{\zeta} + {\mathcal{E}}(\mathit{\zeta} ), {\alpha}\mathit{\delta} + {\mathcal{E}} (\mathit{\delta} ))\big)\\
    		&\preceq  {\lambda} \sum\limits_{j=1}^{m}\nu_j\tilde{h}_j\big( {\mathcal{E}}(\mathit{\zeta} )\big)+(1- {\lambda})  
    		\sum\limits_{j=1}^{m}\nu_j\tilde{h}_j\big({\mathcal{E}} (\mathit{\delta} )\big),\\
    		&= {\lambda} {g}\big({\mathcal{E}}(\mathit{\zeta} )\big)+(1-{\lambda}){g}\big({\mathcal{E}} (\mathit{\delta} )\big).
    	    \end{aligned}
    	\end{equation*}
    	 Thus, the function ${g}$ is SLU$\mathcal{E}$P on $ {\mathcal{S}}$.
    \end{proof}
     
      The concept of prepseudoinvex functions was initially introduced by Jeyakumar \cite{Jeyakumar}. Later, Iqbal and Hussain \cite{Iqbal2022} extended this idea by defining strongly pseudo $\mathcal{\mathcal{E}}$-preinvex functions on S$\mathcal{E}$I sets. Building upon their work, we now propose a new class of interval-valued functions called pseudo strongly LU-$\mathcal{\mathcal{E}}$-preinvex (PSLU$\mathcal{E}$P) functions, as defined below.

       \begin{definition}
     	Let $ {\mathcal{S}}\subseteq  \mathbb{R}^{n}$ be a S$\mathcal{E}$I set. A function $ \tilde{h}: {\mathcal{S}}\subseteq  \mathbb{R}^{n}\rightarrow  \mathcal{I}(\mathbb{R})$ is called PSLU$\mathcal{E}$P with respect to $ {\Psi}$ on $ {\mathcal{S}}$, if $\exists$
  a strictly positive function $ \mathfrak{\Phi}: \mathbb{R}^{n}\times \mathbb{R}^{n}\rightarrow  \mathcal{I}(\mathbb{R})$ such that 
     	
     	\noindent
     	$\tilde{h}\big({\mathcal{E}}(\mathit{\zeta} )\big)\prec\tilde{h}\big({\mathcal{E}} (\mathit{\delta} )\big)$ $\implies$ 
     	
     	\noindent
     	$\tilde{h}\big({\alpha}\mathit{\delta} +{\mathcal{E}} (\mathit{\delta} )+{\lambda}{\Psi}( {\alpha}\mathit{\zeta} +{\mathcal{E}}(\mathit{\zeta} ), {\alpha}\mathit{\delta} +{\mathcal{E}} (\mathit{\delta} ))\big)\preceq\tilde{h}\big({\mathcal{E}} (\mathit{\delta} )\big)+ {\lambda}({\lambda}-1)\mathfrak{\Phi}\big({\mathcal{E}}(\mathit{\zeta} ),{\mathcal{E}} (\mathit{\delta} )\big),$ $\forall~\mathit{\zeta} ,\mathit{\delta}  \in  {\mathcal{S}}, {\alpha}\in {[0,1]}~\&~{\lambda} \in {[0,1]}.$\\
     	
     	\noindent
     	For $\alpha=0$ the IVF $\tilde{h}$ is called PLU$\mathcal{E}$P.\\
        For the real valued function, if  $\alpha=0$ it reduces to pseudo-$\mathcal{E}$-preinvex (P$\mathcal{E}$P) function defined by Iqbal et al. \cite{Iqbal2022}
     \end{definition}

      The following theorem establishes a connection between  SLU$\mathcal{E}$P functions and PSLU$\mathcal{E}$P functions.

    \begin{theorem} \label{theorem 3.4 p4}
    	Let ${\mathcal{S}}\subseteq\mathbb{R}^{n}$ be a S$\mathcal{E}I$ set and 
       $\tilde{h}:{\mathcal{S}}\subseteq\mathbb{R}^{n}\rightarrow  \mathcal{I}(\mathbb{R})$ be a SLU$\mathcal{E}$P function w.r.t. ${\Psi}$ on ${\mathcal{S}}$. Then, $\tilde{h}$ is PSLU$\mathcal{E}$P function on ${\mathcal{S}}$.
    \end{theorem}
    \begin{proof}
    	Let $\tilde{h}\big({\mathcal{E}}(\mathit{\zeta} )\big)\prec \tilde{h}\big({\mathcal{E}} (\mathit{\delta} )\big)$. Since  $ \tilde{h}$ is a SLU$\mathcal{E}$P function on $ {\mathcal{S}}$, for every $\mathit{\zeta} ,\mathit{\delta} \in  {\mathcal{S}}~~ and~~{\alpha},{\lambda}\in {[0,1]} $, we have\\
    	
    	$\tilde{h}\big({\alpha}\mathit{\delta} + {\mathcal{E}}( \mathit{\delta} )+{\lambda}{\Psi}( {\alpha}\mathit{\zeta} + {\mathcal{E}}(\mathit{\zeta} ), {\alpha}\mathit{\delta} + {\mathcal{E}} (\mathit{\delta} ))\big)$
    \begin{equation*}
        \begin{aligned}
             \hspace{1.3cm}&\preceq{\lambda}\tilde{h}\big({\mathcal{E}}(\mathit{\zeta} )\big)+(1-{\lambda})\tilde{h}\big({\mathcal{E}}( \mathit{\delta} )\big)\\ 
    	&=\lambda[\,\tilde{h}^L({\mathcal{E}}(\zeta)),\,\tilde{h}^U({\mathcal{E}}(\zeta))\,]+(1-\lambda)[\,\tilde{h}^L({\mathcal{E}}(\delta)),\,\tilde{h}^U({\mathcal{E}}(\delta))\,]\\
        &=[\lambda\tilde{h}^L({\mathcal{E}}(\zeta))+(1-\lambda)\tilde{h}^L({\mathcal{E}}(\delta)),\lambda\tilde{h}^U({\mathcal{E}}(\zeta))+(1-\lambda)\tilde{h}^U({\mathcal{E}}(\delta))]\\
        &=[\tilde{h}^L({\mathcal{E}}(\delta)),\tilde{h}^U({\mathcal{E}}(\delta))]+[\lambda(\tilde{h}^L({\mathcal{E}}(\zeta))-\tilde{h}^L({\mathcal{E}}(\delta))),\lambda(\tilde{h}^U({\mathcal{E}}(\zeta))-\tilde{h}^U({\mathcal{E}}(\delta)))]\\
        &=\tilde{h}({\mathcal{E}}(\delta))+\lambda[\tilde{h}^L({\mathcal{E}}(\zeta))-\tilde{h}^L({\mathcal{E}}(\delta)),\tilde{h}^U({\mathcal{E}}(\zeta))-\tilde{h}^U({\mathcal{E}}(\delta))]\\
    	&\preceq \tilde{h}\big({\mathcal{E}} (\mathit{\delta} )\big)+{\lambda}(1-{\lambda})[\tilde{h}^L({\mathcal{E}}(\zeta))-\tilde{h}^L({\mathcal{E}}(\delta)),\tilde{h}^U({\mathcal{E}}(\zeta))-\tilde{h}^U({\mathcal{E}}(\delta))]\\
    	&=\tilde{h}\big({\mathcal{E}} (\mathit{\delta} )\big)+{\lambda}({\lambda}-1)[\tilde{h}^U({\mathcal{E}} (\mathit{\delta} ))-\tilde{h}^U({\mathcal{E}}(\mathit{\zeta} )),\tilde{h}^L({\mathcal{E}} (\mathit{\delta} ))-\tilde{h}^L({\mathcal{E}}(\mathit{\zeta} ))]\\
    	&=\tilde{h}\big({\mathcal{E}} (\mathit{\delta} )\big)+{\lambda}({\lambda}-1)\mathfrak{\Phi}\big({\mathcal{E}}(\mathit{\zeta} ),({\mathcal{E}} (\mathit{\delta} )\big).
        \end{aligned}
    \end{equation*}
    
    \noindent
    where, 
    
    \begin{equation*}
\begin{aligned}
\mathfrak{\Phi}\big({\mathcal{E}}(\zeta),\,{\mathcal{E}}(\delta)\big) 
= [\tilde{h}^U({\mathcal{E}} (\mathit{\delta} ))-\tilde{h}^U({\mathcal{E}}(\mathit{\zeta} )),\tilde{h}^L({\mathcal{E}} (\mathit{\delta} ))-\tilde{h}^L({\mathcal{E}}(\mathit{\zeta} ))].
\end{aligned}
\end{equation*}

Therefore, the function $\tilde{h}$ is PSLU$\mathcal{E}$P on ${\mathcal{S}}$.
    \end{proof}
In support of Theorem~\ref{theorem 3.4 p4}, the following illustrative example is provided.

\begin{example}\label{example3.3}
Let $\tilde{h} : (-\infty, 0] \to \mathcal{I}(\mathbb{R})
\quad \text{be an IVF defined by} \quad
\tilde{h}(\zeta) = [\zeta, \zeta + \frac{1}{2}]$ Define a mapping
\[
\mathcal{E} : \mathbb{R} \to \mathbb{R} 
\quad \text{by} \quad 
\mathcal{E}(\zeta) = \lfloor \zeta \rfloor,
\quad \text{where} \; \lfloor \cdot \rfloor \; \text{denotes the greatest integer function.}
\]
Also, let $\Psi : \mathbb{R} \times \mathbb{R} \to \mathbb{R} 
\quad \text{be defined by} \quad 
\Psi(\zeta, \delta) = \zeta - \delta.$\\
Given the IVF $\tilde{h}$, we verify the SLUEP condition:
\[
\tilde{h}\big(\alpha\, \delta + \mathcal{E}(\delta) + \lambda\, \Psi(\alpha\zeta+\mathcal{E}(\zeta), \alpha\, \delta + \mathcal{E}(\delta) \big) 
\preceq \lambda\, \tilde{h}(\mathcal{E}(\zeta)) + (1-\lambda)\, \tilde{h}(\mathcal{E}(\delta)).
\]
This expands to:
\[
\tilde{h}\big(\alpha\, \delta + \lfloor \delta \rfloor + \lambda(\lfloor \zeta \rfloor - \lfloor \delta \rfloor)+\alpha \lambda(\zeta-\delta)\big)
\preceq \lambda\big[\lfloor \zeta \rfloor,\, \lfloor \zeta \rfloor + \tfrac{1}{2} \big]  + (1-\lambda) \big[\lfloor \delta \rfloor,\, \lfloor \delta \rfloor + \tfrac{1}{2}\big].
\]
\[
\tilde{h}\big(\alpha(\lambda \zeta+(1-\lambda)\delta)+\lambda \lfloor \zeta \rfloor+(1-\lambda)\lfloor \delta \rfloor)\preceq \big[\lambda \lfloor \zeta \rfloor+(1-\lambda)\lfloor \delta \rfloor,\lambda \lfloor \zeta \rfloor+(1-\lambda)\lfloor \delta \rfloor+\tfrac{1}{2} \big]
\]
Clearly,
$\alpha(\lambda \zeta+(1-\lambda)\delta)+\lambda \lfloor \zeta \rfloor+(1-\lambda)\lfloor \delta \rfloor 
\leq \lambda \lfloor \zeta \rfloor+(1-\lambda)\lfloor \delta \rfloor$,
and 
\[
\alpha(\lambda \zeta+(1-\lambda)\delta)+\lambda \lfloor \zeta \rfloor+(1-\lambda)\lfloor \delta \rfloor +\tfrac{1}{2}
\leq \lambda \lfloor \zeta \rfloor+(1-\lambda)\lfloor \delta \rfloor+\tfrac{1}{2}.
\]

Thus, $\tilde{h}$ is a SLU$\mathcal{E}$P function.
Next, we verify that $\tilde{h}$ also satisfies the PSLUEP condition:
\begin{equation*}
\begin{aligned}
 &\tilde{h}(\mathcal{E}(\zeta)) \preceq \tilde{h}(\mathcal{E}(\delta))\\ 
\implies&
    \tilde{h}\big(\alpha\, \delta + \mathcal{E}(\delta) + \lambda\, \Psi(\alpha\zeta+\mathcal{E}(\zeta), \alpha\, \delta + \mathcal{E}(\delta) \big) 
\preceq \tilde{h}(\mathcal{E}(\delta)) + \lambda(\lambda-1)\mathfrak{\Phi}\big(\mathcal{E}(\zeta), \mathcal{E}(\delta))\big).   
\end{aligned}
\end{equation*}
Expanding,
\begin{equation*}
    \begin{aligned}
        \tilde{h}\big(\alpha(\lambda \zeta+(1-\lambda)\delta)+\lambda \lfloor \zeta \rfloor+(1-\lambda)\lfloor \delta \rfloor\big)\preceq& \big[\lfloor \delta \rfloor,\lfloor \delta \rfloor+\tfrac{1}{2}\big]+\lambda(\lambda-1)\big[\lfloor \delta \rfloor-\lfloor \zeta \rfloor,\lfloor \delta \rfloor-\lfloor \zeta \rfloor\big]
    \end{aligned}
\end{equation*}
\begin{equation*}
    \begin{aligned}
        &\tilde{h}\big(\alpha(\lambda \zeta+(1-\lambda)\delta)+\lambda \lfloor \zeta \rfloor+(1-\lambda)\lfloor \delta \rfloor\big)\\
        \preceq& \big[\lambda \lfloor \zeta \rfloor+(1-\lambda)\lfloor \delta \rfloor+\lambda^2(\lfloor \delta \rfloor-\lfloor \zeta \rfloor),\lambda \lfloor \zeta \rfloor+(1-\lambda)\lfloor \delta \rfloor+\lambda^2(\lfloor \delta \rfloor-\lfloor \zeta \rfloor)+\tfrac{1}{2} \big]
    \end{aligned}
\end{equation*}
Clearly,
\begin{equation*}
    \begin{aligned}
        \alpha(\lambda \zeta+(1-\lambda)\delta)+\lambda \lfloor \zeta \rfloor+(1-\lambda)\lfloor \delta \rfloor \leq \lambda \lfloor \zeta \rfloor+(1-\lambda)\lfloor \delta \rfloor+\lambda^2(\lfloor \delta \rfloor-\lfloor \zeta \rfloor)
    \end{aligned}
\end{equation*}
also 
\begin{equation*}
    \begin{aligned}
         \alpha(\lambda \zeta+(1-\lambda)\delta)+\lambda \lfloor \zeta \rfloor+(1-\lambda)\lfloor \delta \rfloor+\tfrac{1}{2} \leq \lambda \lfloor \zeta \rfloor+(1-\lambda)\lfloor \delta \rfloor+\lambda^2(\lfloor \delta \rfloor-\lfloor \zeta \rfloor)+\tfrac{1}{2}
    \end{aligned}
\end{equation*}
and hence, $\tilde{h}$ is a PSLU$\mathcal{E}$P function.
\noindent
Therefore, $\tilde{h}$ is both SLU$\mathcal{E}$P and PSLU$\mathcal{E}$P.
\begin{figure}[H]
		\centering
		\includegraphics[width=0.9\linewidth]{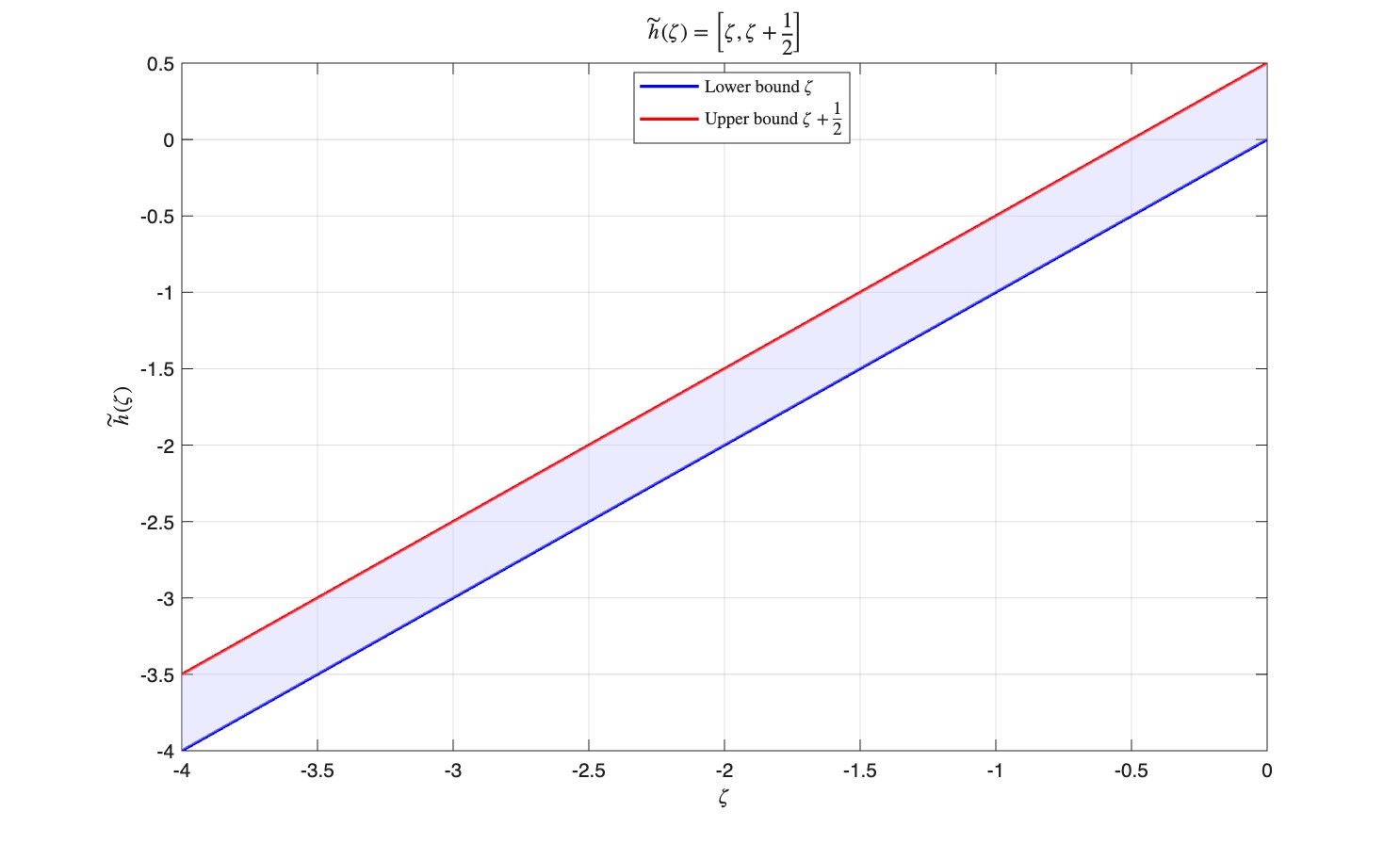}
		\label{fig:placeholder}
		\caption{The function $\tilde{h}$ is an SLU$\mathcal{E}$P as well as PSLU$\mathcal{E}$P with respect to the mappings $\mathcal{E}$ and $\Psi$ given in Example~\ref{example3.3}.
}
\end{figure}

\end{example}
 
    \begin{theorem} 
    	Let ${\mathcal{S}}\subseteq  \mathbb{R}^{n}$ be a S$\mathcal{E}I$ set and $\{ \tilde{h}_j\}_{j\in J}$ be a collection of  functions defined on $ {\mathcal{S}}$ such that $\forall~ \mathit{\zeta} \in {\mathcal{S}},~\sup\limits_{j\in J} ~ \tilde{h}_j(\mathit{\zeta} )$ exists in $\mathcal{I}(\mathbb{R})$. Let $ \tilde{h}: {\mathcal{S}}\rightarrow  \mathcal{I}(\mathbb{R})$ be a function defined by $ \tilde{h}(\mathit{\zeta} )=\sup\limits_{j\in J} \tilde{h}_j(\mathit{\zeta} ),\forall~ \mathit{\zeta} \in  {\mathcal{S}}$. If $ \tilde{h}_j: {\mathcal{S}}\rightarrow  \mathcal{I}(\mathbb{R})$, for every $j\in J$, are  SLU$\mathcal{E}$P functions on $ {\mathcal{S}}$, then  $ \tilde{h}$ is  SLU$\mathcal{E}$P on $ {\mathcal{S}}$.   
    \end{theorem}
    \begin{proof}
    	Suppose that $ \tilde{h}_j: {\mathcal{S}}\rightarrow  \mathcal{I}(\mathbb{R})$, for every $ ~j\in J$, are SLU$\mathcal{E}$P functions on $ {\mathcal{S}}$, then $\forall~\mathit{\zeta} ,\mathit{\delta} \in {\mathcal{S}}$, ${\alpha} \in {[0,1]}~\&~ {\lambda} \in {[0,1]} $, we get\\
    	
    	$\tilde{h}_j\big( {\alpha}  \mathit{\delta} + {\mathcal{E}}( \mathit{\delta} )+ {\lambda} {\Psi}( {\alpha}  \mathit{\zeta} + {\mathcal{E}}( \mathit{\zeta} ), {\alpha}  \mathit{\delta} + {\mathcal{E}}( \mathit{\delta} ))\big)\preceq{\lambda}  \tilde{h}_j\big( {\mathcal{E}}( \mathit{\zeta} )\big)+(1- {\lambda}) \tilde{h}_j\big( {\mathcal{E}}( \mathit{\delta} )\big)$\\
    	
    	$\sup\limits_{j\in J} \tilde{h}_j\big( {\alpha}  \mathit{\delta} + {\mathcal{E}}( \mathit{\delta} )+ {\lambda} {\Psi}( {\alpha}  \mathit{\zeta} + {\mathcal{E}}( \mathit{\zeta} ), {\alpha}  \mathit{\delta} + {\mathcal{E}}( \mathit{\delta} ))\big)$
     \begin{equation*}
         \begin{aligned}
              \hspace{6.2cm}&\preceq{\lambda}\sup\limits_{j\in J}\tilde{h}_j\big({\mathcal{E}}(\mathit{\zeta} )\big)+(1- {\lambda})\sup\limits_{j\in J}\tilde{h}_j\big( {\mathcal{E}} (\mathit{\delta} )\big)\\
     	&= {\lambda}\tilde{h}\big({\mathcal{E}}(\mathit{\zeta} )\big)+(1-{\lambda})\tilde{h}\big({\mathcal{E}} (\mathit{\delta} )\big),\\
         \end{aligned}
     \end{equation*}
     $$\tilde{h}\big({\alpha}\mathit{\delta} + {\mathcal{E}} (\mathit{\delta} )+{\lambda} {\Psi}({\alpha}\mathit{\zeta} + {\mathcal{E}}(\mathit{\zeta} ), {\alpha}\mathit{\delta} +{\mathcal{E}} (\mathit{\delta} ))\big)
     \preceq{\lambda}\tilde{h}\big({\mathcal{E}}(\mathit{\zeta} )\big)+(1-{\lambda})\tilde{h}\big({\mathcal{E}} (\mathit{\delta} )\big).$$
    \end{proof}    

    We now define the notion of a strongly G-invex set in the context of an interval-valued function.
   \begin{definition}
    		\noindent
Let $\mathcal{T} \subseteq \mathbb{R}^{n} \times \mathcal{I}(\mathbb{R})$, where $\mathcal{E} : \mathbb{R}^{n} \to \mathbb{R}^{n}$ and $\mathcal{E}_{0} : \mathcal{I}(\mathbb{R}) \to \mathcal{I}(\mathbb{R})$ are given mappings. The set $\mathcal{T}$ is said to be $G$ strongly $E$-invex (GSEI) with respect to $\Psi$ if $\forall$ $(\zeta, \mathcal{I}_1), (\delta, \mathcal{I}_2) \in \mathcal{T}$ and for any $\alpha, \lambda \in [0,1]$, the following condition holds:
\[
\Big(
\alpha\, \delta + \mathcal{E}(\delta) + \lambda\, \Psi\big(\alpha\, \zeta + \mathcal{E}(\zeta),\, \alpha\, \delta + \mathcal{E}(\delta)\big),\;\;
\lambda\, \mathcal{E}_{0}(\mathcal{I}_1) + (1 - \lambda)\, \mathcal{E}_{0}(\mathcal{I}_2)
\Big)
\in \mathcal{T}.
\]
\end{definition}
\noindent
A set ${\mathcal{S}} \subseteq \mathbb{R}^{n}$ is said to be \textbf{S$\mathcal{E}$I} if and only if ${\mathcal{S}} \times \mathcal{I}(\mathbb{R})$ is a \textbf{GSEI} set with respect to ${\Psi}$.

\noindent
The \emph{epigraph} of $\tilde{h}$, denoted by $\operatorname{epi}(\tilde{h})$, is defined by
\[
\operatorname{epi}(\tilde{h}) = \big\{\, (\zeta,\, \mathcal{I}) \in {\mathcal{S}} \times \mathcal{I}(\mathbb{R}) :\, \tilde{h}(\zeta) \preceq \mathcal{I} \big\}.
\]

\noindent
Next, we present a useful characterization of a SLU$\mathcal{E}$P function in terms of its epigraph.

    	 \begin{theorem} \label{theorem 3.6 p4} 
    	 	\noindent
Let $\mathcal{S} \subseteq \mathbb{R}^{n}$ be a S$\mathcal{E}$I set with respect to $\Psi: \mathbb{R}^{n} \times \mathbb{R}^{n} \to \mathbb{R}^{n}$. Suppose $\tilde{h}: \mathcal{S} \to \mathcal{I}(\mathbb{R})$ is a function and $\mathcal{E}_{0}: \mathcal{I}(\mathbb{R}) \to \mathcal{I}(\mathbb{R})$ is a onto map satisfying
\[
\mathcal{E}_{0}\big( \tilde{h}(\zeta) + t \big) = \tilde{h}\big( \mathcal{E}(\zeta) \big) + t,\quad \forall\, t \geq 0.
\]
Then, $\tilde{h}$ is SLU$\mathcal{E}$P on $\mathcal{S}$ if and only if  $epi(\tilde{h})$ is GSEI set in $\mathcal{S} \times \mathcal{I}(\mathbb{R})$.

    	 \end{theorem} 
    	   
       \begin{proof}
       	Let $(\mathit{\zeta} ,\mathcal{I}_1),(\mathit{\delta} ,\mathcal{I}_2)\in epi(\tilde{h})$. Since ${\mathcal{S}}$ is S$\mathcal{E}I$ set, we have 
       	$$ {\alpha} \mathit{\delta} + {\mathcal{E}} (\mathit{\delta} )+ { \lambda} {\Psi}\big( {\alpha} \mathit{\zeta} + {\mathcal{E}}(\mathit{\zeta} ), {\alpha} \mathit{\delta} + {\mathcal{E}} (\mathit{\delta} )\big)\in{\mathcal{S}},$$ 
      for all $\mathit{\zeta} ,\mathit{\delta} \in {\mathcal{S}},~{\alpha}\in {[0,1]}~\&~{\lambda}\in {[0,1]} $. Let $ {\mathcal{E}}_{0}(\mathcal{I}_1)$ and $ {\mathcal{E}}_{0}(\mathcal{I}_2)$ be such that $$ \tilde{h}\big( {\mathcal{E}}(\mathit{\zeta} )\big)\preceq  {\mathcal{E}}_{0}(\mathcal{I}_1),~ \tilde{h}\big( {\mathcal{E}} (\mathit{\delta} )\big)\preceq  {\mathcal{E}}_{0}(\mathcal{I}_2).$$
       	Then,   $\big( {\mathcal{E}}(\mathit{\zeta} ), {\mathcal{E}}_{0}(\mathcal{I}_1)\big),\big( {\mathcal{E}} (\mathit{\delta} ), {\mathcal{E}}_{0}(\mathcal{I}_2)\big)\in epi( \tilde{h})$.
       	Since the function $\tilde{h}$ is SLUEP on $ {\mathcal{S}}$, we get
       	\begin{equation*}
       	    \begin{aligned}
       	        \tilde{h}\big({\alpha}  \mathit{\delta} + {\mathcal{E}}( \mathit{\delta} )+ {\lambda} {\Psi}( {\alpha}  \mathit{\zeta} + {\mathcal{E}}( \mathit{\zeta} ), {\alpha}  \mathit{\delta} + {\mathcal{E}}( \mathit{\delta} ))\big)&\preceq  {\lambda}  \tilde{h}\big( {\mathcal{E}}( \mathit{\zeta} )\big)+(1- {\lambda}) \tilde{h}\big( {\mathcal{E}}( \mathit{\delta} )\big)\\
       		&\preceq  {\lambda}  {\mathcal{E}}_{0}(\mathcal{I}_1)+(1- {\lambda}) {\mathcal{E}}_{0}(\mathcal{I}_2).
       	    \end{aligned}
       	\end{equation*}
       	Thus, 
       	$$\big({\alpha} \mathit{\delta} + {\mathcal{E}}( \mathit{\delta} )+ {\lambda} {\Psi}( {\alpha}  \mathit{\zeta} + {\mathcal{E}}( \mathit{\zeta} ), {\alpha}  \mathit{\delta} +{\mathcal{E}}( \mathit{\delta} )), {\lambda}{\mathcal{E}}_{0}(\mathcal{I}_1)+(1- {\lambda}) {\mathcal{E}}_{0}(\mathcal{I}_2)\big)\in epi(\tilde{h}).$$
       	Therefore, the set $epi(\tilde{h})$ is GSEI on $ {\mathcal{S}}\times \mathcal{I}(\mathbb{R})$.
       	
       		\vspace{.2cm}
       	\noindent
       	Conversely, let $epi( \tilde{h})$ be a GSEI set on $ {\mathcal{S}}\times  \mathcal{I}(\mathbb{R})$. Let $\mathit{\zeta} ,\mathit{\delta} \in  {\mathcal{S}},~ {\alpha} \in {[0,1]}~\&~ {\lambda}\in {[0,1]} $. Then, $(\mathit{\zeta} , \tilde{h}(\mathit{\zeta} )),(\mathit{\delta} , \tilde{h} (\mathit{\delta} ))\in epi(\tilde{h})$. Since $epi( \tilde{h})$ is GSEI on ${\mathcal{S}}\times \mathcal{I}(\mathbb{R})$, we have
       	$$\big({\alpha}\mathit{\delta} +{\mathcal{E}} (\mathit{\delta} )+ {\lambda} {\Psi}( {\alpha}  \mathit{\zeta} + {\mathcal{E}}( \mathit{\zeta} ), {\alpha}\mathit{\delta} + {\mathcal{E}} (\mathit{\delta} )), ~{\lambda}  {\mathcal{E}}_{0}( \tilde{h}( \mathit{\zeta} ))+(1-{\lambda}) {\mathcal{E}}_{0}( \tilde{h}( \mathit{\delta} ))\big)\in epi( \tilde{h}),$$ which implies that
       	\begin{equation*}
       	    \begin{aligned}
       	         \tilde{h}\big({\alpha}  \mathit{\delta} + {\mathcal{E}}( \mathit{\delta} )+ {\lambda} {\Psi}( {\alpha}  \mathit{\zeta} + {\mathcal{E}}( \mathit{\zeta} ), {\alpha}  \mathit{\delta} + {\mathcal{E}}( \mathit{\delta} ))\big)&\preceq {\lambda}{\mathcal{E}}_{0}\big(\tilde{h}(\mathit{\zeta} )\big)+(1-{\lambda}){\mathcal{E}}_{0}\big( \tilde{h}( \mathit{\delta} )\big),\\
       		&\preceq {\lambda}\tilde{h}\big({\mathcal{E}}( \mathit{\zeta} )\big)+(1-{\lambda})\tilde{h}\big({\mathcal{E}}( \mathit{\delta} )\big).
       	    \end{aligned}
       	\end{equation*}
       Thus, $\tilde{h}$ is SLUEP on $ {\mathcal{S}}$.    
       \end{proof}

       \begin{definition}
    An IVF $\tilde{h} : \mathcal{I}(\mathbb{R})  \to \mathcal{I}(\mathbb{R}) $ is said to be LU non-decreasing if the following holds:
\[
\tilde{h}(\mathcal{I}_1) \preceq \tilde{h}(\mathcal{I}_2), \quad \text{whenever } \mathcal{I}_1\preceq \mathcal{I}_2,
\]
\end{definition}
\begin{example}
    For an IVF $\tilde{h} : \mathcal{ I(\mathbb{R})}  \to \mathcal{ I(\mathbb{R})} $ defined as 
    $\tilde{h}(\mathcal{O})=\alpha\mathcal{O} $, where $\alpha \geq 0$ then $\tilde{h}$ is LU non-decreasing.
\end{example}

       \begin{theorem}
       	Let ${\mathcal{S}} \subseteq \mathbb{R}^{n}$ be an S$\mathcal{E}$I set and let $\tilde{h} : \mathbb{R}^{n} \to \mathcal{I}(\mathbb{R})$ be a SLU$\mathcal{E}$P function with respect to ${\Psi}$ on ${\mathcal{S}}$. Assume that ${\varphi} : \mathcal{I}(\mathbb{R}) \to \mathcal{I}(\mathbb{R})$ is positively homogeneous and LU non-decreasing function. Then the composition ${\varphi} \circ \tilde{h}$ is also a SLU$\mathcal{E}$P function on ${\mathcal{S}}$.
       \end{theorem}
       \begin{proof} The result follows directly from the definitions.

       \end{proof}
       \begin{theorem} \label{theorem 3.8}
Suppose that for each $j = 1, \dots, m$, the function $g_j : \mathbb{R}^{n} \to \mathcal{I}(\mathbb{R})$ is SLU$\mathcal{E}$P with respect to $\Psi$. If the mapping $\mathcal{E}$ satisfies $\mathcal{E}(\mathcal{S}) \subseteq \mathcal{S}$, then the set
\[
\mathcal{S} = \big\{\, \zeta \in \mathbb{R}^{n} :\, g_j(\zeta) \preceq 0 , \forall ~j = 1, \dots, m \big\}
\]
is an S$\mathcal{E}$I set.
\end{theorem}

\begin{proof}
Since each $g_j$ is SLU$\mathcal{E}$P on $\mathbb{R}^{n}$, for any $\zeta,\, \delta \in \mathcal{S}$, $\alpha \in [0,1]$, and $\lambda \in [0,1]$, we have
\[
g_j\!\big(\alpha\, \delta + \mathcal{E}(\delta) + \lambda\, \Psi(\alpha\, \zeta + \mathcal{E}(\zeta),\, \alpha\, \delta + \mathcal{E}(\delta))\big)
\preceq
\lambda\, g_j(\mathcal{E}(\zeta)) + (1 - \lambda)\, g_j(\mathcal{E}(\delta)).
\]
Now, by the assumption $\mathcal{E}(\mathcal{S}) \subseteq \mathcal{S}$ and the definition of $\mathcal{S}$, it follows that
\[
g_j(\mathcal{E}(\zeta)) \preceq 0 \quad \text{and} \quad g_j(\mathcal{E}(\delta)) \preceq 0.
\]
Therefore, $g_j\!\big(\alpha\, \delta + \mathcal{E}(\delta) + \lambda\, \Psi(\alpha\, \zeta + \mathcal{E}(\zeta),\, \alpha\, \delta + \mathcal{E}(\delta))\big) \preceq 0,
\quad \forall\, 1 \leq j \leq m.$
This shows that
$
\alpha\, \delta + \mathcal{E}(\delta) + \lambda\, \Psi(\alpha\, \zeta + \mathcal{E}(\zeta),\, \alpha\, \delta + \mathcal{E}(\delta)) \in \mathcal{S}.
$
Hence, by Definition \ref{definition 2.9 p4}, $\mathcal{S}$ is S$\mathcal{E}$I.
\end{proof}

\noindent

By using Theorem~3.1 and definition of the $gH$-product from \cite{Tauheed}, we obtain the following lemma.
\begin{lemma}\label{lemma 3.2 p4} 
    Let $\tilde{h}:\mathcal{S}\subseteq\mathbb{R}^n\rightarrow\mathcal{I}(\mathbb{R})$, \( \mathit{\nu}  \in \mathbb{R}^n \), $
    \nabla \tilde{h}(\mathit{\zeta} ) = \left(\frac{\partial \tilde{h}}{\partial \mathit{\zeta}_1}, \frac{\partial \tilde{h}}{\partial \mathit{\zeta}_2} , \ldots, \frac{\partial \tilde{h}}{\partial \mathit{\zeta}_n} \right)^T.$\\
    and $\frac{\partial \tilde{h}^L}{\partial\mathit{\zeta}_i}~ \textit{and}~\frac{\partial \tilde{h}^U}{\partial\mathit{\zeta}_i}$ exist. Then,
    \begin{equation*}
       \langle v, \nabla \tilde{h}(\mathit{\zeta}) \rangle_{gH} =
\begin{cases}
[ \nu \nabla \tilde{h}^L(\mathit{\zeta}),\, \nu \nabla \tilde{h}^U(\mathit{\zeta}) ], & \text{if } \nu \nabla \tilde{h}^L(\mathit{\zeta}) \leq \nu \nabla \tilde{h}^U(\mathit{\zeta}) \\
[ \nu \nabla \tilde{h}^U(\mathit{\zeta}),\, \nu \nabla \tilde{h}^L(\mathit{\zeta}) ], & \text{if } \nu \nabla \tilde{h}^U(\mathit{\zeta}) \leq \nu \nabla \tilde{h}^L(\mathit{\zeta})
\end{cases} 
    \end{equation*}
    
    \begin{proof}:
        \begin{equation*}
        \begin{aligned}
            \langle \mathit{\nu}, \nabla\tilde{h}(\mathit{\zeta}) \rangle_{gH} &= \sum_{i \in j^+} \mathit{\nu}_i \frac{\partial \tilde{h}}{\partial \mathit{\zeta}_i} \,\ominus_{gH}\, \sum_{k \in j^-} |\mathit{\nu}_k| \frac{\partial \tilde{h}}{\partial \mathit{\zeta}_k}.\\
        &= \left[ \sum_{i \in j^+} \mathit{\nu}_i \frac{\partial \tilde{h}^L}{\partial \mathit{\zeta}_i} , \sum_{i \in j^+} \mathit{\nu}_i \frac{\partial \tilde{h}^U}{\partial \mathit{\zeta}_i} \right]
\,\ominus_{gH}\,
\left[ \sum_{k \in j^-} |\mathit{\nu}_k| \frac{\partial \tilde{h}^L}{\partial \mathit{\zeta}_k} , \sum_{k \in j^-} |\mathit{\nu}_k| \frac{\partial \tilde{h}^U}{\partial \mathit{\zeta}_k} \right].\\
&= \left[
\min \{ p, q \},\, \max \{ p, q \}
\right],
\quad \\
&\text{where}
\quad
p = \sum_{i \in j^+} \mathit{\nu}_i \frac{\partial \tilde{h}^L}{\partial \mathit{\zeta}_i} - \sum_{k \in j^-} |\mathit{\nu}_k| \frac{\partial \tilde{h}^L}{\partial \mathit{\zeta}_k},
\quad
q = \sum_{i \in j^+} \mathit{\nu}_i \frac{\partial \tilde{h}^U}{\partial \mathit{\zeta}_i} - \sum_{k \in j^-} |\mathit{\nu}_k| \frac{\partial \tilde{h}^U}{\partial \mathit{\zeta}_k}.
        \end{aligned}
    \end{equation*}
\textbf{Case 1:} \( p \leq q \)

\begin{equation*}
    \begin{aligned}
        \langle \mathit{\nu}, \nabla\tilde{h}(\mathit{\zeta}) \rangle_{gh}
&= \left[
\sum_{i \in j^+} \mathit{\nu}_i \frac{\partial \tilde{h}^L}{\partial \mathit{\zeta}_i} - \sum_{k \in j^-} |\mathit{\nu}_k| \frac{\partial \tilde{h}^L}{\partial \mathit{\zeta}_k},\sum_{i \in j^+} \mathit{\nu}_i \frac{\partial \tilde{h}^U}{\partial \mathit{\zeta}_i} - \sum_{k \in j^-} |\mathit{\nu}_k| \frac{\partial \tilde{h}^U}{\partial \mathit{\zeta}_k}
\right]\\
&= [\mathit{\nu}\nabla\tilde{h}^L(\mathit{\zeta}),\mathit{\nu}\nabla\tilde{h}^U(\mathit{\zeta})],
    \end{aligned}
\end{equation*}

\textbf{Case 2:} \( p > q \)

\[
\langle \mathit{\nu}, \nabla\tilde{h}(\mathit{\zeta}) \rangle_{gH} = [ \mathit{\nu}\nabla\tilde{h}^U(\mathit{\zeta}),\, \mathit{\nu}\nabla\tilde{h}^L(\mathit{\zeta}) ].
\]
    
    \end{proof}

\end{lemma}
        
      Now, we define the strongly LU-$\mathcal{E}$-invex (SLU$\mathcal{E}$I) and weakly S$\mathcal{E}I$ function on S$\mathcal{E}I$ set as follows:  
\begin{definition}  \label{definition 3.6 p4}
    Let $\mathcal{S} \subseteq \mathbb{R}^n$ be a S$\mathcal{\mathcal{E}}$I set w.r.t. $ {\Psi}$. Suppose $\tilde{h} : \mathcal{S} \rightarrow \mathcal{I}(\mathbb{R})$ is a weakly differentiable IVF, expressed as $\tilde{h}(\mathit{\zeta} ) = [\tilde{h}^L(\mathit{\zeta} ), \tilde{h}^U(\mathit{\zeta} )]$, then $\tilde{h}$ is said to be weakly S$\mathcal{\mathcal{E}}$I functions w.r.t. $ {\Psi}$ on $ \mathcal{S}$ if, $\forall$ $\mathit{\zeta} ,\mathit{\delta}  \in \mathcal{S}$, the following inequalities hold:
\begin{equation} 
    \begin{aligned}
        \tilde{h}^L(\mathcal{\mathcal{E}}(\mathit{\zeta} )) - \tilde{h}^L(\mathcal{\mathcal{E}}(\mathit{\delta} )) 
        &\geq \Psi({\alpha} \mathit{\zeta} + \mathcal{\mathcal{E}}(\mathit{\zeta} )), {\alpha}\mathit{\delta} + \mathcal{\mathcal{E}}(\mathit{\delta} )))^{\!\top}\nabla \tilde{h}^L(\mathcal{\mathcal{E}}(\mathit{\delta} )),
    \end{aligned}
\end{equation}
\begin{equation}
    \begin{aligned}
        \tilde{h}^U(\mathcal{\mathcal{E}}(\mathit{\zeta} )) - \tilde{h}^U(\mathcal{\mathcal{E}}(\mathit{\delta} )) 
        &\geq \Psi({\alpha} \mathit{\zeta} + \mathcal{\mathcal{E}}(\mathit{\zeta} )), {\alpha}\mathit{\delta} + \mathcal{\mathcal{E}}(\mathit{\delta} )))^{\!\top}\nabla \tilde{h}^U(\mathcal{\mathcal{E}}(\mathit{\delta} )).
    \end{aligned}
\end{equation}

\end{definition}

\begin{definition} \label{definition 3.7 p4}
Let $\mathcal{S}\subseteq  \mathbb{R}^{n}$ be an open S$\mathcal{\mathcal{E}}$I set and  $\tilde{h}\colon \mathcal{S}\rightarrow \mathcal{I}(\mathbb{R})$ be a gH differentiable function on $\mathcal{S}$. Then, $\tilde{h}$ is  called  SLU$\mathcal{\mathcal{E}}$I w.r.t. $ {\Psi}$ on $ \mathcal{S}$, if
\begin{eqnarray*}
		\langle\ {\Psi}( {\alpha} \mathit{\zeta} + \mathcal{\mathcal{E}}(\mathit{\zeta} ), {\alpha}\mathit{\delta} + \mathcal{\mathcal{E}}(\mathit{\delta} )),\nabla  \tilde{h}( \mathcal{\mathcal{E}}(\mathit{\delta} ))\rangle_{gH}\preceq  \tilde{h}( \mathcal{\mathcal{E}}(\mathit{\zeta} )))\ominus_{gH} \tilde{h}( \mathcal{\mathcal{E}}(\mathit{\delta} ))),
\end{eqnarray*}
  $\forall \mathit{\zeta} ,\mathit  t\in \mathcal {S}$ and $ {\alpha}\in {[0,1]} $.\\
 In the context of real-valued functions, when $\alpha = 0$, the function $\tilde{h}$ reduces to  $\mathcal{E}$-invex function introduced by Jaiswal et al. \cite{Jaiswal}. 
\end{definition} 

The following theorem demonstrates that every weakly S$\mathcal{\mathcal{E}}$I function is also SLU$\mathcal{\mathcal{E}}$I.
\begin{theorem} \label{theorem 3.9 p4}
Let $\mathcal{S}$ be a S$\mathcal{\mathcal{E}}$I subset of $\mathbb{R}^n$  with respect to the mapping $\Psi : \mathcal{S} \times \mathcal{S} \rightarrow \mathbb{R}^n$. If a weakly differentiable IVF $\tilde{h}(\mathit{\zeta} ) = [\tilde{h}^L(\mathit{\zeta} ), \tilde{h}^U(\mathit{\zeta} )]$ is weakly S$\mathcal{\mathcal{E}}$I function, then it is SLU$\mathcal{\mathcal{E}}$I.
\begin{proof}
   Because $\tilde{h} $ is weakly S$\mathcal{\mathcal{E}}$I, then the by Definition \ref{definition 3.6 p4}, we have

\begin{align}
\label{eq 3.3 p2}\tilde{h}^L(\mathcal{\mathcal{E}}(\mathit{\zeta} )) - \tilde{h}^L(\mathcal{\mathcal{E}}(\mathit{\delta} )) &\geq \Psi(\alpha \mathit{\zeta} +\mathcal{\mathcal{E}}(\mathit{\zeta} ),\alpha \mathit{\delta} +\mathcal{\mathcal{E}}(\mathit{\delta} ))^{\!\top} \nabla \tilde{h}^L(\mathcal{\mathcal{E}}(\mathit{\delta} )), \\
\label{eq 3.4 p2}\tilde{h}^U(\mathcal{\mathcal{E}}(\mathit{\zeta} )) - \tilde{h}^U(\mathcal{\mathcal{E}}(\mathit{\delta} )) &\geq \Psi(\alpha \mathit{\zeta} +\mathcal{\mathcal{E}}(\mathit{\zeta} ),\alpha \mathit{\delta} +\mathcal{\mathcal{E}}(\mathit{\delta} ))^{\!\top} \nabla \tilde{h}^U(\mathcal{\mathcal{E}}(\mathit{\delta} )),
\end{align}

for each $\lambda \in [0,1]$ and each $\mathit{\zeta}  \in \mathcal{S}$.\\
\textbf{Case (I)}:
\begin{equation*}
    \Psi(\alpha \mathit{\zeta} +\mathcal{\mathcal{E}}(\mathit{\zeta} ),\alpha \mathit{\delta} +\mathcal{\mathcal{E}}(\mathit{\delta} ))^{\!\top} \nabla \tilde{h}^L(\mathcal{\mathcal{E}}(\mathit{\delta} )) \leq \Psi(\alpha \mathit{\zeta} +\mathcal{\mathcal{E}}(\mathit{\zeta} ),\alpha \mathit{\delta} +\mathcal{\mathcal{E}}(\mathit{\delta} ))^{\!\top} \nabla \tilde{h}^U(\mathcal{\mathcal{E}}(\mathit{\delta} )),
\end{equation*} 
From Lemma \ref{lemma 3.2 p4} we have,

\begin{equation*}
\begin{aligned}
    \langle\ {\Psi}( {\alpha} \mathit{\zeta} + \mathcal{\mathcal{E}}(\mathit{\zeta} ), {\alpha}\mathit{\delta} + \mathcal{\mathcal{E}}(\mathit{\delta} )),\nabla  \tilde{h}( \mathcal{\mathcal{E}}(\mathit{\delta} ))\rangle_{gH}
    = \big[
    &\Psi(\alpha \mathit{\zeta}  + \mathcal{\mathcal{E}}(\mathit{\zeta} ),\, \alpha \mathit{\delta}  + \mathcal{\mathcal{E}}(\mathit{\delta} ))^{\!\top}
    \nabla \tilde{h}^L(\mathcal{\mathcal{E}}(\mathit{\delta} )), \\
    &\Psi(\alpha \mathit{\zeta}  + \mathcal{\mathcal{E}}(\mathit{\zeta} ),\, \alpha \mathit{\delta}  + \mathcal{\mathcal{E}}(\mathit{\delta} ))^{\!\top}
    \nabla \tilde{h}^U(\mathcal{\mathcal{E}}(\mathit{\delta} )) 
    \big].
\end{aligned}
\end{equation*}

Now,
\begin{equation*}
\begin{aligned}
\tilde{h}(\mathcal{E}(\zeta)) \ominus_{gH} \tilde{h}(\mathcal{E}(\delta))
= \Big[
    &\min\!\Big\{
        \tilde{h}^L(\mathcal{E}(\zeta)) - \tilde{h}^L(\mathcal{E}(\delta)),\,
        \tilde{h}^U(\mathcal{E}(\zeta)) - \tilde{h}^U(\mathcal{E}(\delta))
    \Big\}, \\
    &\max\!\Big\{
        \tilde{h}^L(\mathcal{E}(\zeta)) - \tilde{h}^L(\mathcal{E}(\delta)),\,
        \tilde{h}^U(\mathcal{E}(\zeta)) - \tilde{h}^U(\mathcal{E}(\delta))
    \Big\}
\Big].
\end{aligned}
\end{equation*}

If, $\tilde{h}^L(\mathcal{\mathcal{E}}(\mathit{\zeta} )) - \tilde{h}^L(\mathcal{\mathcal{E}}(\mathit{\delta} )) \leq \tilde{h}^U(\mathcal{\mathcal{E}}(\mathit{\zeta} )) - \tilde{h}^U(\mathcal{\mathcal{E}}(\mathit{\delta} ))$
Then,
\begin{equation*}
   \tilde{h}(\mathcal{\mathcal{E}}(\mathit{\zeta} )) \ominus_{gH}\tilde{h}(\mathcal{\mathcal{E}}(\mathit{\delta} )) = [\tilde{h}^L(\mathcal{\mathcal{E}}(\mathit{\zeta} )) - \tilde{h}^L(\mathcal{\mathcal{E}}(\mathit{\delta} )), \tilde{h}^U(\mathcal{\mathcal{E}}(\mathit{\zeta} )) - \tilde{h}^U(\mathcal{\mathcal{E}}(\mathit{\delta} ))],
\end{equation*} 
then from (\ref{eq 3.3 p2}) and (\ref{eq 3.4 p2}), we have

\begin{equation*}
   \tilde{h}(\mathcal{\mathcal{E}}(\mathit{\zeta} )) \ominus_{gH}\tilde{h}(\mathcal{\mathcal{E}}(\mathit{\delta} )) \succeq \langle\ {\Psi}( {\alpha} \mathit{\zeta} + \mathcal{\mathcal{E}}(\mathit{\zeta} ), {\alpha}\mathit{\delta} + \mathcal{\mathcal{E}}(\mathit{\delta} )),\nabla  \tilde{h}( \mathcal{\mathcal{E}}(\mathit{\delta} ))\rangle_{gH}.
\end{equation*}
Hence $\tilde{h}$ is SLU$\mathcal{\mathcal{E}}$I.\\
\vspace{0.1cm}

Now, If, $\tilde{h}^L(\mathcal{\mathcal{E}}(\mathit{\zeta} )) - \tilde{h}^L(\mathcal{\mathcal{E}}(\mathit{\delta} )) \geq \tilde{h}^U(\mathcal{\mathcal{E}}(\mathit{\zeta} )) - \tilde{h}^U(\mathcal{\mathcal{E}}(\mathit{\delta} ))$
Then,
\begin{equation*}
   \tilde{h}(\mathcal{\mathcal{E}}(\mathit{\zeta} )) \ominus_{gH}\tilde{h}(\mathcal{\mathcal{E}}(\mathit{\delta} )) = [\tilde{h}^U(\mathcal{\mathcal{E}}(\mathit{\zeta} )) - \tilde{h}^U(\mathcal{\mathcal{E}}(\mathit{\delta} )), \tilde{h}^L(\mathcal{\mathcal{E}}(\mathit{\zeta} )) - \tilde{h}^L(\mathcal{\mathcal{E}}(\mathit{\delta} ))],
\end{equation*} 
then

\begin{equation*}
    \begin{aligned}
        \tilde{h}^L(\mathcal{\mathcal{E}}(\mathit{\zeta} )) - \tilde{h}^L(\mathcal{\mathcal{E}}(\mathit{\delta} )) 
&\geq \tilde{h}^U(\mathcal{\mathcal{E}}(\mathit{\zeta} )) - \tilde{h}^U(\mathcal{\mathcal{E}}(\mathit{\delta} )) \\
&\geq \Psi(\alpha \mathit{\zeta} +\mathcal{\mathcal{E}}(\mathit{\zeta} ),\alpha \mathit{\delta} +\mathcal{\mathcal{E}}(\mathit{\delta} ))^{\!\top}\nabla \tilde{h}^U(\mathcal{\mathcal{E}}(\mathit{\delta} )) \\
&\geq \Psi(\alpha \mathit{\zeta} +\mathcal{\mathcal{E}}(\mathit{\zeta} ),\alpha \mathit{\delta} +\mathcal{\mathcal{E}}(\mathit{\delta} ))^{\!\top}\nabla \tilde{h}^L(\mathcal{\mathcal{E}}(\mathit{\delta} )).
    \end{aligned}
\end{equation*}
Thus,
\begin{equation*}
   \tilde{h}(\mathcal{\mathcal{E}}(\mathit{\zeta} )) \ominus_{gH}\tilde{h}(\mathcal{\mathcal{E}}(\mathit{\delta} )) \succeq \langle\ {\Psi}( {\alpha} \mathit{\zeta} + \mathcal{\mathcal{E}}(\mathit{\zeta} ), {\alpha}\mathit{\delta} + \mathcal{\mathcal{E}}(\mathit{\delta} )),\nabla  \tilde{h}( \mathcal{\mathcal{E}}(\mathit{\delta} ))\rangle_{gH}.
\end{equation*}

\textbf{Case (II)}:
\begin{equation*}
    \Psi(\alpha \mathit{\zeta} +\mathcal{\mathcal{E}}(\mathit{\zeta} ),\alpha \mathit{\delta} +\mathcal{\mathcal{E}}(\mathit{\delta} ))^{\!\top}\nabla \tilde{h}^L(\mathcal{\mathcal{E}}(\mathit{\delta} )) \geq \Psi(\alpha \mathit{\zeta} +\mathcal{\mathcal{E}}(\mathit{\zeta} ),\alpha \mathit{\delta} +\mathcal{\mathcal{E}}(\mathit{\delta} ))^{\!\top}\nabla \tilde{h}^U(\mathcal{\mathcal{E}}(\mathit{\delta} )),
\end{equation*} 
From Lemma \ref{lemma 3.2 p4} we have,

\begin{equation*}
\begin{aligned}
    \langle\ {\Psi}( {\alpha} \mathit{\zeta} + \mathcal{\mathcal{E}}(\mathit{\zeta} ), {\alpha}\mathit{\delta} + \mathcal{\mathcal{E}}(\mathit{\delta} )),\nabla  \tilde{h}( \mathcal{\mathcal{E}}(\mathit{\delta} ))\rangle_{gH}
    = \big[
    &\Psi(\alpha \mathit{\zeta}  + \mathcal{\mathcal{E}}(\mathit{\zeta} ),\, \alpha \mathit{\delta}  + \mathcal{\mathcal{E}}(\mathit{\delta} ))^{\!\top}
    \nabla \tilde{h}^U(\mathcal{\mathcal{E}}(\mathit{\delta} )), \\
    &\Psi(\alpha \mathit{\zeta}  + \mathcal{\mathcal{E}}(\mathit{\zeta} ),\, \alpha \mathit{\delta}  + \mathcal{\mathcal{E}}(\mathit{\delta} ))^{\!\top}
    \nabla \tilde{h}^L(\mathcal{\mathcal{E}}(\mathit{\delta} )) 
    \big].
\end{aligned}
\end{equation*}

If 
\begin{equation*}
   \tilde{h}(\mathcal{\mathcal{E}}(\mathit{\zeta} )) \ominus_{gH}\tilde{h}(\mathcal{\mathcal{E}}(\mathit{\delta} )) = [\tilde{h}^U(\mathcal{\mathcal{E}}(\mathit{\zeta} )) - \tilde{h}^U(\mathcal{\mathcal{E}}(\mathit{\delta} )), \tilde{h}^L(\mathcal{\mathcal{E}}(\mathit{\zeta} )) - \tilde{h}^L(\mathcal{\mathcal{E}}(\mathit{\delta} ))],
\end{equation*} 
then (\ref{eq 3.3 p2}) and (\ref{eq 3.4 p2}), we have

\begin{equation*}
   \tilde{h}(\mathcal{\mathcal{E}}(\mathit{\zeta} )) \ominus_{gH}\tilde{h}(\mathcal{\mathcal{E}}(\mathit{\delta} )) \succeq \langle\ {\Psi}( {\alpha} \mathit{\zeta} + \mathcal{\mathcal{E}}(\mathit{\zeta} ), {\alpha}\mathit{\delta} + \mathcal{\mathcal{E}}(\mathit{\delta} )),\nabla  \tilde{h}( \mathcal{\mathcal{E}}(\mathit{\delta} ))\rangle_{gH}.
\end{equation*}

If 
\begin{equation*}
   \tilde{h}(\mathcal{\mathcal{E}}(\mathit{\zeta} )) \ominus_{gH}\tilde{h}(\mathcal{\mathcal{E}}(\mathit{\delta} ) ) = [\tilde{h}^L(\mathcal{\mathcal{E}}(\mathit{\zeta} )) - \tilde{h}^L(\mathcal{\mathcal{E}}(\mathit{\delta} )), \tilde{h}^U(\mathcal{\mathcal{E}}(\mathit{\zeta} )) - \tilde{h}^U(\mathcal{\mathcal{E}}(\mathit{\delta} ))],
\end{equation*} 
then

\begin{equation*}
    \begin{aligned}
        \tilde{h}^U(\mathcal{\mathcal{E}}(\mathit{\zeta} )) - \tilde{h}^U(\mathcal{\mathcal{E}}(\mathit{\delta} )) 
&\geq \tilde{h}^L(\mathcal{\mathcal{E}}(\mathit{\zeta} )) - \tilde{h}^L(\mathcal{\mathcal{E}}(\mathit{\delta} )) \\
&\geq \Psi(\alpha \mathit{\zeta} +\mathcal{\mathcal{E}}(\mathit{\zeta} ),\alpha \mathit{\delta} +\mathcal{\mathcal{E}}(\mathit{\delta} ))^{\!\top}\nabla \tilde{h}^L(\mathcal{\mathcal{E}}(\mathit{\delta} )) \\
&\geq \Psi(\alpha \mathit{\zeta} +\mathcal{\mathcal{E}}(\mathit{\zeta} ),\alpha \mathit{\delta} +\mathcal{\mathcal{E}}(\mathit{\delta} ))^{\!\top}\nabla \tilde{h}^U(\mathcal{\mathcal{E}}(\mathit{\delta} )).
    \end{aligned}
\end{equation*}
Thus,
\begin{equation*}
   \tilde{h}(\mathcal{\mathcal{E}}(\mathit{\zeta} )) \ominus_{gH}\tilde{h}(\mathcal{\mathcal{E}}(\mathit{\delta} )) \succeq \langle\ {\Psi}( {\alpha} \mathit{\zeta} + \mathcal{\mathcal{E}}(\mathit{\zeta} ), {\alpha}\mathit{\delta} + \mathcal{\mathcal{E}}(\mathit{\delta} )),\nabla  \tilde{h}( \mathcal{\mathcal{E}}(\mathit{\delta} ))\rangle_{gH}.
\end{equation*}

\end{proof}
\end{theorem}
The following example validate the Theorem \ref{theorem 3.9 p4}

 \begin{example}
     Consider an IVF $\tilde{h}: \mathcal{S} \to \mathcal{I}(\mathbb{R})$ defined on S$\mathcal{\mathcal{E}}$I set $\mathcal{S}=[ln(2),\infty)$,
   \begin{equation*}
       \tilde{h}(\mathit{\zeta} )=[4\mathit{\zeta} -8ln(\mathit{\zeta} ),8\mathit{\zeta} -16ln(\mathit{\zeta} )] , ~~~\forall \mathit{\zeta}  \in \mathcal{S},
   \end{equation*}   
and let $\mathcal{\mathcal{E}}:\mathbb{R} \to\mathbb{R}$ be a map defined as $\mathcal{\mathcal{E}}(\mathit{\zeta} ) = e^{\mathit{\zeta} }$ $\forall$ $\mathit{\zeta}  \in \mathbb{R}$.
Also, let $\Psi: \mathbb{R} \times \mathbb{R} \to \mathbb{R}$ be defined as:
    $\Psi(\mathit{\zeta} ,\mathit{\delta} )=-1$.\\ 
Then, by Definition~\ref{definition 3.6 p4}, $\tilde{h}(\zeta)$ is weakly S$\mathcal{\mathcal{E}}$I. Also, by Definition \ref{definition 3.7 p4}, the function $\tilde{h}(\zeta)$ is SLU$\mathcal{\mathcal{E}}$I function. 
 \end{example} 
Although Theorem~\ref{theorem 3.9 p4} establishes a significant result, its converse is not valid in all situations. To illustrate this, we now present the following counterexample.

\begin{example} \label{example3.6}
    Let $\tilde{h} : \mathbb{R} \rightarrow \mathcal{I}(\mathbb{R})$ be IVF defined by 
$\tilde{h}(\zeta) = [-|\zeta|,\, |\zeta|]$. 
Consider the mapping $\mathcal{E} : \mathbb{R} \to \mathbb{R}$ given by $\mathcal{E}(\zeta) = -1$. 
Furthermore, let $\Psi : \mathbb{R} \times \mathbb{R} \to \mathbb{R}$ be defined as
\[
\Psi(\zeta, \delta) =
\begin{cases}
    \zeta + \delta, & \text{if } \zeta,\, \delta \geq 0 \text{ or } \zeta,\, \delta \leq 0, \\
    -\delta, & \text{otherwise}.
\end{cases}
\]
It can be observed that $\tilde{h}$ exhibits the properties of an SLU$\mathcal{E}$I function, whereas it fails to meet the criteria of a weakly S$\mathcal{E}$I function.

\begin{figure}[H]
		\centering
		\includegraphics[width=0.9\linewidth]{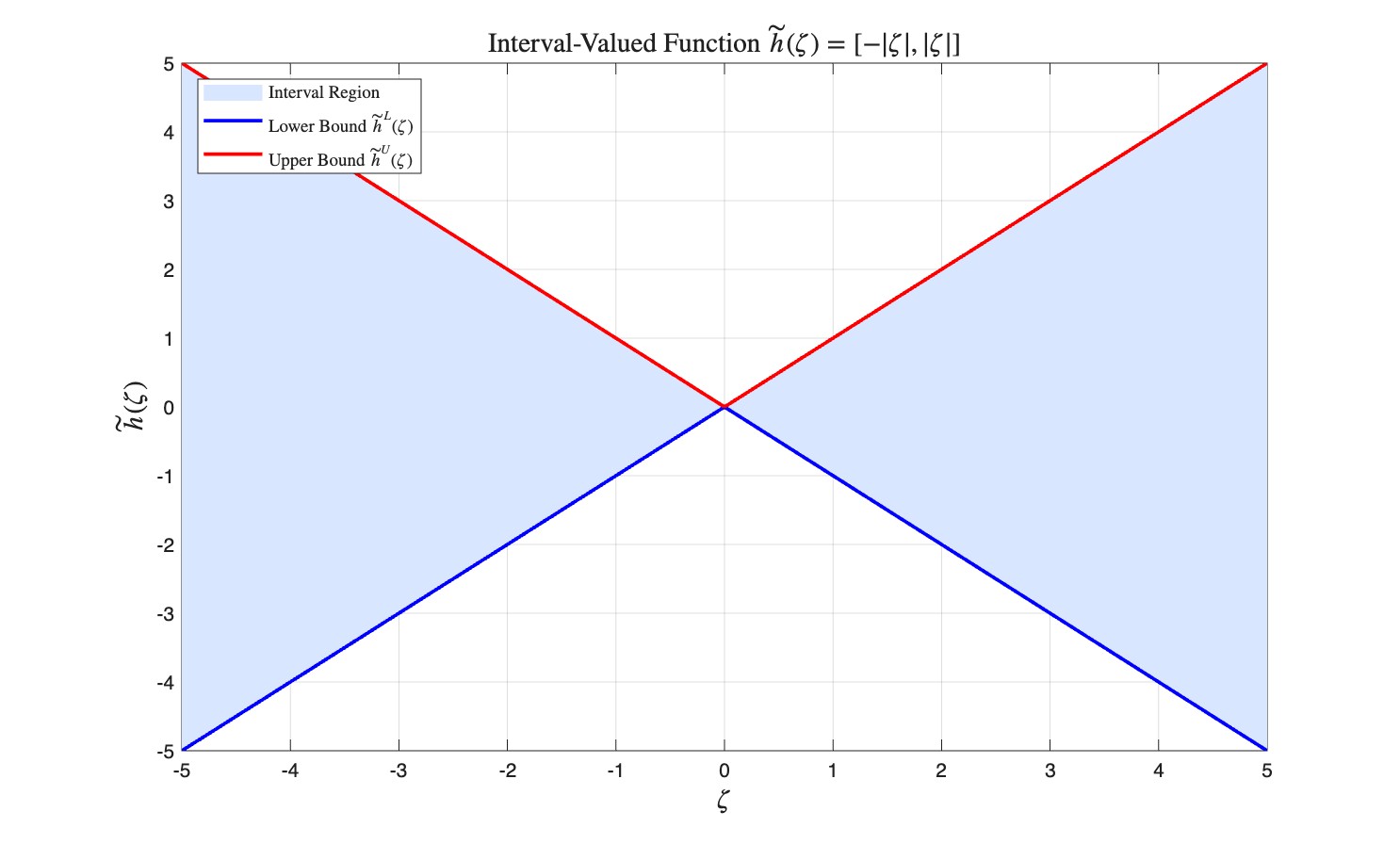}
		\label{fig:placeholder}
		\caption{The function $\tilde{h}$ is SLU$\mathcal{E}$I but not weakly S$\mathcal{E}$I with respect to the mappings $\mathcal{E}$ and $\Psi$ given in Example~\ref{example3.6}.
}
\end{figure}

\end{example}

     Inspired by the work of Mohan et al. \cite{Neogy} and Kumari et al. \cite{Babli}, Iqbal and Hussain \cite{Iqbal2022} proposed the notion of $\mathbf{Condition\ A}$, defined as follows:

     	\noindent
\textbf{Condition A.}\label{Condition A} 
Let ${\mathcal{E}}: {\mathcal{S}} \to {\mathcal{S}}$ be an onto mapping, where ${\mathcal{S}} \subseteq \mathbb{R}^{n}$ is a S$\mathcal{E}$I set with respect to ${\Psi}: \mathbb{R}^{n} \times \mathbb{R}^{n} \to \mathbb{R}^{n}$. Assume that for every $\zeta,\, \delta \in {\mathcal{S}}$, and for any ${\alpha}, {\lambda} \in [0,1]$, there exists $\bar{\delta} \in {\mathcal{S}}$ such that
\[
{\mathcal{E}}(\bar{\delta}) = {\alpha}\, \delta + {\mathcal{E}}(\delta) + {\lambda}\, {\Psi}\big({\alpha}\, \zeta + {\mathcal{E}}(\zeta),\, {\alpha}\, \delta + {\mathcal{E}}(\delta)\big) \in {\mathcal{S}}.
\]
The mapping ${\Psi}$ is said to satisfy \textbf{Condition A} if the following hold simultaneously:
\[
\quad
{\Psi}\big({\alpha}\, \delta + {\mathcal{E}}(\delta),\, {\alpha}\, \bar{\delta} + {\mathcal{E}}(\bar{\delta})\big)
= -{\lambda}\Big({\alpha}\, \bar{\delta} + {\Psi}\big({\alpha}\, \zeta + {\mathcal{E}}(\zeta),\, {\alpha}\, \delta + {\mathcal{E}}(\delta)\big)\Big),
\]
\[
\quad
{\Psi}\big({\alpha}\, \zeta + {\mathcal{E}}(\zeta),\, {\alpha}\, \bar{\delta} + {\mathcal{E}}(\bar{\delta})\big)
= (1 - {\lambda})\Big({\alpha}\, \bar{\delta} + {\Psi}\big({\alpha}\, \zeta + {\mathcal{E}}(\zeta),\, {\alpha}\, \delta + {\mathcal{E}}(\delta)\big)\Big).
\]

\medskip
\noindent
If ${\alpha} = 0$ and ${\mathcal{E}}(\zeta) = \zeta$, then \textbf{Condition A} reduces to the \textbf{Condition C} given by Mohan and Neogy~\cite{Neogy}.

\medskip
\noindent
In the following theorem, we present a significant relationship between weakly S$\mathcal{E}$I function and SLU$\mathcal{E}$P function.
\begin{theorem}
	Let $\mathcal{S} \subseteq \mathbb{R}^{n}$ be an open S$\mathcal{E}$I set with respect to the mapping $\Psi$ and a mapping $\mathcal{E} : \mathcal{S} \to \mathcal{S}$. Suppose that the function $\tilde{h} : \mathcal{S} \to \mathcal{I}(\mathbb{R})$ is weakly differentiable and satisfies the weakly S$\mathcal{E}$I condition with respect to $\Psi$ on $\mathcal{S}$. If, in addition, $\Psi$ satisfies \textbf{Condition A}, then $\tilde{h}$ is a SLU$\mathcal{E}$P function with respect to $\Psi$ on $\mathcal{S}$.
\end{theorem}

\begin{proof}
    Let $\zeta,\, \delta \in \mathcal{S}$. Since $\mathcal{S}$ is an S$\mathcal{E}$I set and $\mathcal{E}$ is onto, there exists $\bar{\delta} \in \mathcal{S}$ such that
    \begin{equation}\label{eq:bar_delta}
        \mathcal{E}(\bar{\delta}) 
        = \alpha\, \delta + \mathcal{E}(\delta) + \lambda\, \Psi\big(\alpha \zeta + \mathcal{E}(\zeta),\, \alpha \delta + \mathcal{E}(\delta)\big)
        \in \mathcal{S},
        \quad \forall\, \alpha, \lambda \in [0,1].
    \end{equation}

    Since $\tilde{h}$ is weakly S$\mathcal{E}$I, we have:
    \begin{align}
        \tilde{h}^L(\mathcal{E}(\zeta)) - \tilde{h}^L(\mathcal{E}(\bar{\delta}))
        &\geq 
        \Psi\big(\alpha \zeta + \mathcal{E}(\zeta),\, \alpha \bar{\delta} + \mathcal{E}(\bar{\delta})\big)^{\!\top}
        \nabla \tilde{h}^L(\mathcal{E}(\bar{\delta})), 
        \label{eq:ineq1L} \\
        \tilde{h}^U(\mathcal{E}(\zeta)) - \tilde{h}^U(\mathcal{E}(\bar{\delta}))
        &\geq 
        \Psi\big(\alpha \zeta + \mathcal{E}(\zeta),\, \alpha \bar{\delta} + \mathcal{E}(\bar{\delta})\big)^{\!\top}
        \nabla \tilde{h}^U(\mathcal{E}(\bar{\delta})),
        \label{eq:ineq1U}
    \end{align}
    and also,
    \begin{align}
        \tilde{h}^L(\mathcal{E}(\delta)) - \tilde{h}^L(\mathcal{E}(\bar{\delta}))
        &\geq 
        \Psi\big(\alpha \delta + \mathcal{E}(\delta),\, \alpha \bar{\delta} + \mathcal{E}(\bar{\delta})\big)^{\!\top}
        \nabla \tilde{h}^L(\mathcal{E}(\bar{\delta})),
        \label{eq:ineq2L} \\[6pt]
        \tilde{h}^U(\mathcal{E}(\delta)) - \tilde{h}^U(\mathcal{E}(\bar{\delta}))
        &\geq 
        \Psi\big(\alpha \delta + \mathcal{E}(\delta),\, \alpha \bar{\delta} + \mathcal{E}(\bar{\delta})\big)^{\!\top}
        \nabla \tilde{h}^U(\mathcal{E}(\bar{\delta})).
        \label{eq:ineq2U}
    \end{align}

    Multiplying \eqref{eq:ineq1L} by $\lambda$ and \eqref{eq:ineq2L} by $1-\lambda$, then adding, gives:
    \begin{align}
        &\lambda\big[\tilde{h}^L(\mathcal{E}(\zeta)) - \tilde{h}^L(\mathcal{E}(\bar{\delta}))\big] 
        + (1-\lambda)\big[\tilde{h}^L(\mathcal{E}(\delta)) - \tilde{h}^L(\mathcal{E}(\bar{\delta}))\big] \notag\\
        &\quad \geq 
        \Big[
        \lambda\, \Psi\big(\alpha \zeta + \mathcal{E}(\zeta),\, \alpha \bar{\delta} + \mathcal{E}(\bar{\delta})\big) 
        + (1-\lambda)\, \Psi\big(\alpha \delta + \mathcal{E}(\delta),\, \alpha \bar{\delta} + \mathcal{E}(\bar{\delta})\big)
        \Big]^{\!\top} 
        \nabla \tilde{h}^L(\mathcal{E}(\bar{\delta})). 
        \label{eq:combineL}
    \end{align}

    Using \textbf{Condition A}, we have
    \[
    \lambda\, \Psi\big(\alpha \zeta + \mathcal{E}(\zeta),\, \alpha \bar{\delta} + \mathcal{E}(\bar{\delta})\big) 
    + (1-\lambda)\, \Psi\big(\alpha \delta + \mathcal{E}(\delta),\, \alpha \bar{\delta} + \mathcal{E}(\bar{\delta})\big)
    = 0.
    \]

    Therefore, \eqref{eq:combineL} reduces to:
    \[
    \tilde{h}^L(\mathcal{E}(\bar{\delta})) 
    \leq \lambda\, \tilde{h}^L(\mathcal{E}(\zeta)) + (1-\lambda)\, \tilde{h}^L(\mathcal{E}(\delta)).
    \]

    Using \eqref{eq:bar_delta}, this implies:
    \[
    \tilde{h}^L\!\Big(
    \alpha \delta + \mathcal{E}(\delta) 
    + \lambda\, \Psi(\alpha \zeta + \mathcal{E}(\zeta),\, \alpha \delta + \mathcal{E}(\delta))
    \Big) 
    \leq 
    \lambda\, \tilde{h}^L(\mathcal{E}(\zeta)) + (1-\lambda)\, \tilde{h}^L(\mathcal{E}(\delta)).
    \]

    A similar argument using \eqref{eq:ineq1U} and \eqref{eq:ineq2U} shows:
    \[
    \tilde{h}^U\!\Big(
    \alpha \delta + \mathcal{E}(\delta) 
    + \lambda\, \Psi(\alpha \zeta + \mathcal{E}(\zeta),\, \alpha \delta + \mathcal{E}(\delta))
    \Big) 
    \leq 
    \lambda\, \tilde{h}^U(\mathcal{E}(\zeta)) + (1-\lambda)\, \tilde{h}^U(\mathcal{E}(\delta)).
    \]

    Thus, by Theorem \ref{theorem 3.1 p4}, $\tilde{h}$ is SLU$\mathcal{E}$P with respect to ${\Psi}$ on $\mathcal{S}$.
\end{proof}
\section{\textbf{SLUEP programming problem}}
       
       In this section, we explore a nonlinear programming problem that serves as an application of the SLU$\mathcal{E}$P functions framework. This particular problem, known as the SLU$\mathcal{E}$P programming problem, builds upon and generalizes the findings presented in~\cite{Iqbal2022}.

         \begin{eqnarray}\label{4.1}
          \hspace{3cm}(P)~ 
         \left\{\begin{array}{ll}
         	 min ~~\tilde{h}_{0}(\mathit{\zeta} ),\\
         	 
         	 \tilde{h}_{j}(\mathit{\zeta} )\preceq 0,~~~~1\leq j\leq m,\\
         	 
         	\mathit{\zeta} \in  \mathbb{R}^{n},
         \end{array}\right.
           \end{eqnarray}
            where $ \tilde{h}_{0}: \mathbb{R}^{n}\rightarrow  \mathcal{I}(\mathbb{R})$ and $ \tilde{h}_{j}: \mathbb{R}^{n}\rightarrow  \mathcal{I}(\mathbb{R}),~1\leq j\leq m$, are  SLU$\mathcal{E}$P functions on $ \mathbb{R}^{n}$.
        
        \noindent    
      Let ${X}$ denote the non-empty set of feasible solutions, defined by  
       \begin{equation}\label{4.2} 
          {X}=\big\{\mathit{\zeta} \in\mathbb{R}^{n}~:~\tilde{h}_{j}(\mathit{\zeta} )\preceq 0,~1\leq j\leq m\big\}.
       \end{equation} 

\begin{remark}
   By Theorem~\ref{theorem 3.8}, if the mapping $\mathcal{E}$ satisfies $\mathcal{E}(\mathcal{S}) \subseteq \mathcal{S}$, then the feasible set $X$ is $S\mathcal{E}I$.
\end{remark}





\begin{theorem}
Let $\tilde{h}_{j}: \mathbb{R}^{n} \rightarrow \mathcal{I}(\mathbb{R})$, for $0 \leq j \leq m$, be SLU$\mathcal{E}$P functions on $\mathbb{R}^{n}$, and let $X$ be the non-empty feasible region defined by \emph{(\ref{4.2})}.Consider the set of optimal solutions 
\[
X_{\text{opt}} = \big\{\, \zeta \in X : \tilde{h}_{0}(\zeta) = \min_{w \in X} \tilde{h}_{0}(w) \big\}
\] and $\mathcal{E}(X_{\text{opt}})\subseteq X_{\text{opt}}$. Then, $X_{\text{opt}}$
is an S$\mathcal{E}$I set with respect to $\Psi$.
\end{theorem}

\begin{proof}
By assumption, every $\tilde{h}_{j}$ for $0 \leq j \leq m$ is SLU$\mathcal{E}$P on $\mathbb{R}^{n}$. Let $\mathit{\zeta}, \mathit{\delta} \in X_{\text{opt}}$ and note that, by definition of optimality, we have
\[
\tilde{h}_{0}(\mathit{\zeta}) = \tilde{h}_{0}(\mathit{\delta}) = \tilde{h}_{0}(\mathcal{E}(\mathit{\zeta})) = \tilde{h}_{0}(\mathcal{E}(\mathit{\delta})) = \tilde{h}^{*},
\quad \text{where} \quad 
\tilde{h}^{*} := \min_{w \in X} \tilde{h}_{0}(w).
\]

Since $X$ is feasible and closed under $\mathcal{E}$, it follows that
$\mathit{\zeta}, \mathit{\delta}, \mathcal{E}(\mathit{\zeta}), \mathcal{E}(\mathit{\delta}) \in X$. By the SLU$\mathcal{E}$P property, for any $\alpha, \lambda \in [0,1]$, we have
\[
\tilde{h}_{0}\big( 
\alpha\, \mathit{\delta} + \mathcal{E}(\mathit{\delta})
+ \lambda\, \Psi(
\alpha\, \mathit{\zeta} + \mathcal{E}(\mathit{\zeta}),
\, \alpha\, \mathit{\delta} + \mathcal{E}(\mathit{\delta})
)
\big)
\preceq
\lambda\, \tilde{h}_{0}(\mathcal{E}(\mathit{\zeta})) + (1-\lambda)\, \tilde{h}_{0}(\mathcal{E}(\mathit{\delta})) 
= \tilde{h}^{*}.
\]
But since $\tilde{h}^{*}$ is the minimum value, the above inequality must hold as equality. Therefore,
\[
\tilde{h}_{0}\big(
\alpha\, \mathit{\delta} + \mathcal{E}(\mathit{\delta}) 
+ \lambda\, \Psi(
\alpha\, \mathit{\zeta} + \mathcal{E}(\mathit{\zeta}),
\, \alpha\, \mathit{\delta} + \mathcal{E}(\mathit{\delta})
)
\big) = \tilde{h}^{*}.
\]

Moreover, for all $1 \leq j \leq m$, the constraint functions satisfy
\[
\tilde{h}_{j}\big(
\alpha\, \mathit{\delta} + \mathcal{E}(\mathit{\delta})
+ \lambda\, \Psi(
\alpha\, \mathit{\zeta} + \mathcal{E}(\mathit{\zeta}),
\, \alpha\, \mathit{\delta} + \mathcal{E}(\mathit{\delta})
)
\big) \preceq 0.
\]

Hence, the constructed point lies in $X$ and achieves the optimal value $\tilde{h}^{*}$ for the objective. Therefore, it is also an element of $X_{\text{opt}}$:
\[
\alpha\, \mathit{\delta} + \mathcal{E}(\mathit{\delta})
+ \lambda\, \Psi(
\alpha\, \mathit{\zeta} + \mathcal{E}(\mathit{\zeta}),
\, \alpha\, \mathit{\delta} + \mathcal{E}(\mathit{\delta})
) \in X_{\text{opt}}.
\]
This proves that $X_{\text{opt}}$ is closed under the same generalized operation that defines an S$\mathcal{E}$I set.

Therefore, $X_{\text{opt}}$ is indeed an S$\mathcal{E}$I set with respect to $\Psi$.
\end{proof} 
       	
       	\begin{theorem}   
       		Let $\tilde{h}_{0}: \mathbb{R}^{n} \rightarrow \mathcal{I}(\mathbb{R})$ be a \emph{strictly} SLU$\mathcal{E}$P function on $\mathbb{R}^{n}$, and let each $\tilde{h}_{j}: \mathbb{R}^{n} \rightarrow \mathcal{I}(\mathbb{R})$ for $1 \leq j \leq m$ be SLU$\mathcal{E}$P functions defined on $\mathbb{R}^{n}$. Suppose the feasible region
\[
X = \big\{\, \zeta \in \mathbb{R}^{n} : \tilde{h}_{j}(\zeta) \preceq 0, \; 1 \leq j \leq m \big\}
\]
is non-empty. If $\zeta^{*}$ is a local minimizer of problem~\emph{(\ref{4.1})}, then $\zeta^{*}$ is a strict global minimizer of~\emph{(\ref{4.1})}.
       	\end{theorem}
       	\noindent
\textbf{Proof.}  
Since $\zeta^*$ is a local minimum of problem (4.1), it follows that $\zeta^* \in \mathbb{R}^n$ and $\tilde{h}_j(\zeta^*) \preceq 0$ for all $1 \leq j \leq m$. By the definition of local optimality, there exists some $\epsilon > 0$ such that
\[
\tilde{h}_0(\zeta^*) \preceq \tilde{h}_0(\delta), \quad \forall\, \delta \in B_\epsilon(\zeta^*) \cap X,\; \delta \neq \zeta^*,
\]
where $B_\epsilon(\zeta^*) = \{\delta \in \mathbb{R}^n : \|\delta - \zeta^*\| < \epsilon\}$.

\vspace{0.2cm}
\noindent
Assume, for contradiction, that there exists another feasible point $w^* \in X$, with $w^* \neq \zeta^*$, satisfying
\[
\tilde{h}_0(w^*) \prec \tilde{h}_0(\zeta^*).
\]

Since $\zeta^*, w^* \in X$, by Lemma \ref{lemma 2.1 p4} there exist $\zeta, w \in X$ such that
\[
\zeta^* = \mathcal{E}(\zeta), \quad w^* = \mathcal{E}(w).
\]

By Theorem~\ref{theorem 3.8}, for any $\alpha, \lambda \in [0,1]$,
\[
\alpha \zeta + \zeta^* + \lambda\, \Psi(\alpha w + w^*,\, \alpha \zeta + \zeta^*) \in X.
\]

If $\alpha w + w^* \neq \alpha \zeta + \zeta^*$, then the strict SLU$\mathcal{E}$P property implies
\[
\tilde{h}_0\!\big(\alpha \zeta + \zeta^* + \lambda\, \Psi(\alpha w + w^*,\, \alpha \zeta + \zeta^*)\big)
\prec \lambda\, \tilde{h}_0(w^*) + (1-\lambda)\, \tilde{h}_0(\zeta^*) =: \tilde{h}_0^*.
\]

\vspace{0.2cm}
\noindent
\textbf{Case 1:} $\alpha = 0$ and $\Psi(w^*, \zeta^*) = 0$.  
Then $\alpha \zeta + \zeta^* + \lambda\, \Psi(w^*, \zeta^*) = \zeta^*$, which implies
\[
\tilde{h}_0(\zeta^*) \prec \tilde{h}_0^*,
\]
a contradiction.

\vspace{0.2cm}
\noindent
\textbf{Case 2:} $\alpha = 0$ and $\Psi(w^*, \zeta^*) \neq 0$.  
Let $\bar{\lambda} = \min\!\big\{1,\; \epsilon/\|\Psi(w^*, \zeta^*)\|\big\}$. For any $\lambda \in (0, \bar{\lambda})$,
\[
\big\|\, \zeta^* + \lambda\, \Psi(w^*, \zeta^*) - \zeta^* \big\| = \lambda\, \|\Psi(w^*, \zeta^*)\| < \epsilon.
\]
Thus,
\[
\zeta^* + \lambda\, \Psi(w^*, \zeta^*) \in B_\epsilon(\zeta^*) \cap X,\quad \text{and} \quad \tilde{h}_0\!\big(\zeta^* + \lambda\, \Psi(w^*, \zeta^*)\big) \prec \tilde{h}_0^*,
\]
which contradicts the local optimality of $\zeta^*$.

\vspace{0.2cm}
\noindent
\textbf{Case 3:} $\alpha \neq 0$ and $\Psi(\alpha w + w^*,\, \alpha \zeta + \zeta^*) = 0$.  
Define $\bar{\alpha} = \min\!\big\{1,\; \epsilon/\|\zeta\|\big\}$. For any $\alpha \in (0, \bar{\alpha})$,
\[
\|\alpha \zeta + \zeta^* - \zeta^*\| = \alpha\, \|\zeta\| < \epsilon.
\]
Therefore,
\[
\alpha \zeta + \zeta^* \in B_\epsilon(\zeta^*) \cap X,\quad \text{and} \quad \tilde{h}_0(\alpha \zeta + \zeta^*) \prec \tilde{h}_0^*,
\]
which contradicts the local minimality of $\zeta^*$.

\vspace{0.2cm}
\noindent
\textbf{Case 4:} $\alpha \neq 0$ and $\Psi(\alpha w + w^*,\, \alpha \zeta + \zeta^*) \neq 0$.  
Take $\epsilon_1, \epsilon_2 > 0$ with $\epsilon_1 + \epsilon_2 = \epsilon$. Let $\bar{\alpha} = \min\!\big\{1,\; \epsilon_1/\|\zeta\|\big\}$ and choose any $\alpha \in (0, \bar{\alpha})$. Also, let
\[
\bar{\lambda} = \min\!\Big\{1,\; \frac{\epsilon_2}{\|\Psi(\frac{\epsilon_1}{\|\zeta\|} w + w^*,\, \frac{\epsilon_1}{\|\zeta\|} \zeta + \zeta^*)\|}\Big\},
\quad \lambda \in (0, \bar{\lambda}).
\]
Then,
\[
\| {\alpha} \mathit{\zeta} + \mathit{\zeta}^{*} + {\Psi}\big( {\alpha} w + w^{*},\, {\alpha} \mathit{\zeta} + \mathit{\zeta}^{*} \big) - \mathit{\zeta}^{*} \|
\]

\begin{equation*}
    \begin{aligned}
        &\leq {\alpha} \| \mathit{\zeta} \| + {\lambda} \| {\Psi} \big( {\alpha} w + w^{*},\, {\alpha} \mathit{\zeta} + \mathit{\zeta}^{*} \big) \| \\ 
&< \bar{{\alpha}} \| \mathit{\zeta} \| + {\lambda} \| {\Psi} \big( \bar{{\alpha}} w + w^{*},\, \bar{{\alpha}} \mathit{\zeta} + \mathit{\zeta}^{*} \big) \| \\ 
&\leq \epsilon_{1} + {\lambda} \left\| {\Psi} \left( \frac{\epsilon_{1}}{\| \mathit{\zeta} \|} w + w^{*},\, \frac{\epsilon_{1}}{\| \mathit{\zeta} \|} \mathit{\zeta} + \mathit{\zeta}^{*} \right) \right\|\\
&\leq\epsilon_{1}+\bar{ {\lambda}}\left\| {\Psi}\left(\frac{\epsilon_{1}}{\|\mathit{\zeta}\|} w+w^{*},\frac{\epsilon_{1}}{\|\mathit{\zeta}\|} \mathit{\zeta}+\mathit{\zeta}^{*}\right)\right\|.\\
       		&\leq\epsilon_{1}+\epsilon_{2}=\epsilon.
    \end{aligned}
\end{equation*}

Thus, the constructed point remains in $B_\epsilon(\zeta^*) \cap X$ and
\[
\tilde{h}_0\!\big(\alpha \zeta + \zeta^* + \lambda\, \Psi(\alpha w + w^*,\, \alpha \zeta + \zeta^*)\big) \prec \tilde{h}_0^*,
\]
which again contradicts local optimality.

\vspace{0.2cm}
\noindent
In every case, we reach a contradiction. Therefore, no such $w^*$ exists, and hence $\zeta^*$ is a strict global minimum point of problem (4.1).

\hfill$\square$


Now we define the KKT conditons for SLU$\mathcal{E}$I functions by considering an optimization problem involving an interval-valued objective function and real-valued constraint functions:
\begin{align*}
    \text{min~} \tilde{h}(\mathit{\zeta})& =  [\tilde{h}^L(\mathit{\zeta}), \tilde{h}^U(\mathit{\zeta})]\\
    \text{subject to~}   g_i (\mathit{\zeta}) &\leq 0,\quad i = 1,2,\ldots,n.
\end{align*}
\noindent
Now, we introduce an optimization problem defined on S$\mathcal{\mathcal{E}}$I set.  $\tilde{h}$ is SLU$\mathcal{\mathcal{E}}$I objective function with respect to a function $\Psi : \mathcal{S} \times \mathcal{S} \to \mathbb{R}^n$. Also assuming that inequality constraints $g_i(\mathit{\zeta})$, for $i = 1, 2, \ldots, n$, are S$\mathcal{\mathcal{E}}$I with respect to $\Psi$:

\begin{equation*}
    \begin{aligned}
        \text{(IVOP)} \quad & \min \,\tilde{h}(\mathcal{\mathcal{E}}(\mathit{\zeta} )) = \big[ \tilde{h}^L(\mathcal{\mathcal{E}}(\mathit{\zeta} )), \tilde{h}^U(\mathcal{\mathcal{E}}(\mathit{\zeta} )) \big] \\
& \text{subject to } g_i(\mathcal{\mathcal{E}}(\mathit{\zeta} )) \leq 0, \quad i = 1, 2, \ldots, n.
    \end{aligned}
\end{equation*}

Let the feasible set be defined as
\begin{equation*}
    P = \big\{ \mathit{\zeta}  \in \mathbb{R}^n : g_i(\mathcal{\mathcal{E}}(\mathit{\zeta} )) \leq 0, \, i = 1, 2, \ldots, n \big\}.
\end{equation*}

\begin{definition} \label{definition 4.1 p4}
 Let $\mathit{\zeta} ^*$ be a feasible solution to the IVOP. We say that $\mathit{\zeta} ^*$ is a \emph{non-dominated} solution if there exists no $\mathit{\zeta}  \in P$ such that $\tilde{h}(\mathcal{\mathcal{E}}(\mathit{\zeta} )) \prec \tilde{h}(\mathcal{\mathcal{E}}(\mathit{\zeta} ^*))$. In such a case, $\tilde{h}(\mathcal{\mathcal{E}}(\mathit{\zeta} ^*))$ is termed a non-dominated value of the IVF $\tilde{h}$.   
\end{definition}
Next,we prove the sufficiency of KKT conditions as follows:
\begin{theorem} \label{theorem 4.3 p4}
    Suppose $\tilde{h}:\mathcal{S}\subseteq \mathbb{R}^n\rightarrow \mathcal{I}(\mathbb{R})$ is a  SLU$\mathcal{\mathcal{E}}$I function with respect to $\Psi : \mathcal{S} \times \mathcal{S} \to \mathbb{R}^n$ and  each $g_i(\mathit{\zeta} )$, for $i = 1, 2, \ldots, n$, is S$\mathcal{\mathcal{E}}$I with respect to $\Psi$. If $\mathit{\zeta} ^* \in P$, satisfying the following conditions.
    
    \begin{equation*}
        \begin{aligned}
             \nabla\tilde{h}^L(\mathcal{E}(\mathit{\zeta} ^*))  + \sum_{i=1}^n v_i \nabla g_i(\mathcal{E}(\mathit{\zeta} ^*)) &= 0, \\
             \nabla\tilde{h}^U(\mathcal{E}(\mathit{\zeta} ^*)) + \sum_{i=1}^n v_i \nabla g_i(\mathcal{E}(\mathit{\zeta} ^*)) &= 0, \\
    \sum_{i=1}^n v_i g_i(\mathcal{E}(\mathit{\zeta} ^*)) &= 0, \\
    v_i &\geq 0,
        \end{aligned}
    \end{equation*}
    
    then $\mathit{\zeta} ^*$ is a solution of IVOP which is non-dominated.

  \begin{proof}
     Consider $\zeta^* \in P$. By using the  strongly LU-$\mathcal{E}$-invexity of $\tilde{h}$ and strongly $\mathcal{E}$-invexity of $g_i$, we have 
    
   \begin{equation*}
       \begin{aligned}
           &[\min \left\{
    \tilde{h}^L(\mathcal{E}(\zeta)) - \tilde{h}^L(\mathcal{E}(\zeta^*)),\,
    \tilde{h}^U(\mathcal{E}(\zeta)) - \tilde{h}^U(\mathcal{E}(\zeta^*))
\right\}, \\
&\max \left\{
    \tilde{h}^L(\mathcal{E}(\zeta)) - \tilde{h}^L(\mathcal{E}(\zeta^*)),\,
    \tilde{h}^U(\mathcal{E}(\zeta)) - \tilde{h}^U(\mathcal{E}(\zeta^*))
\right\}]\\
&\succeq  \langle\ {\Psi}( {\alpha} \mathit{\zeta} + \mathcal{\mathcal{E}}(\mathit{\zeta} ), {\alpha}\mathit{\zeta^*} + \mathcal{\mathcal{E}}(\mathit{\zeta^*} )),\nabla  \tilde{h}( \mathcal{\mathcal{E}}(\mathit{\zeta^*} ))\rangle_{gH}.
       \end{aligned}
   \end{equation*}
 In particular, when
   \begin{equation*}
    \begin{aligned}
        \min\{\tilde{h}^L(\mathcal{E}(\zeta)) - \tilde{h}^L(\mathcal{E}(\zeta^*)),~& \tilde{h}^U(\mathcal{E}(\zeta)) - \tilde{h}^U(\mathcal{E}(\zeta^*))\} = \tilde{h}^L(\mathcal{E}(\zeta)) - \tilde{h}^L(\mathcal{E}(\zeta^*)),
\end{aligned}
\end{equation*}
and
\begin{equation*}
\begin{aligned}
    \langle\ {\Psi}( {\alpha} \mathit{\zeta} + \mathcal{\mathcal{E}}(\mathit{\zeta} ), {\alpha}\mathit{\zeta^*} + \mathcal{\mathcal{E}}(\mathit{\zeta^*} )),\nabla  \tilde{h}( \mathcal{\mathcal{E}}(\mathit{\zeta^*} ))\rangle_{gH}
    = \big[
    &\Psi(\alpha \mathit{\zeta}  + \mathcal{\mathcal{E}}(\mathit{\zeta} ),\, \alpha \mathit{\zeta^*}  + \mathcal{\mathcal{E}}(\mathit{\zeta^*} ))^{\!\top}
    \nabla \tilde{h}^L(\mathcal{\mathcal{E}}(\mathit{\zeta^*} )), \\
    &\Psi(\alpha \mathit{\zeta}  + \mathcal{\mathcal{E}}(\mathit{\zeta} ),\, \alpha \mathit{\zeta^*}  + \mathcal{\mathcal{E}}(\mathit{\zeta^*} ))^{\!\top}
    \nabla \tilde{h}^U(\mathcal{\mathcal{E}}(\mathit{\zeta^*} )) 
    \big].
\end{aligned}
\end{equation*}
We get,
\begin{align*}
           \tilde{h}^L(\mathcal{\mathcal{E}}(\mathit{\zeta} )) - \tilde{h}^L(\mathcal{\mathcal{E}}(\mathit{\zeta} ^*)) 
&\geq \Psi(\alpha \mathit{\zeta} +\mathcal{\mathcal{E}}(\mathit{\zeta} ), \alpha \mathit{\zeta} ^*+\mathcal{\mathcal{E}}(\mathit{\zeta} ^*))^{\!\top}\nabla\tilde{h}^L(\mathcal{\mathcal{E}}(\mathit{\zeta} ^*)) \\
&= -\Psi(\alpha \mathit{\zeta} +\mathcal{\mathcal{E}}(\mathit{\zeta} ), \alpha \mathit{\zeta} ^*+\mathcal{\mathcal{E}}(\mathit{\zeta} ^*))^{\!\top}\sum_{i=1}^n v_i \nabla g_i(\mathcal{\mathcal{E}}(\mathit{\zeta} ^*)) \\
&\geq -\sum_{i=1}^n v_i \left( g_i(\mathcal{\mathcal{E}}(\mathit{\zeta} )) - g_i(\mathcal{\mathcal{E}}(\mathit{\zeta} ^*)) \right)\\
&= \sum_{i=1}^{n} \left( -v_i g_i(\mathcal{\mathcal{E}}(\mathit{\zeta} )) + v_i g_i(\mathcal{\mathcal{E}}(\mathit{\zeta} ^*)) \right) \\
&= \sum_{i=1}^{n} -v_i g_i(\mathcal{\mathcal{E}}(\mathit{\zeta} )) \\
&\geq 0.
\end{align*}
and,
\begin{align*}
\tilde{h}^U(\mathcal{\mathcal{E}}(\mathit{\zeta} )) - \tilde{h}^U(\mathcal{\mathcal{E}}(\mathit{\zeta} ^*)) &\geq \Psi(\alpha \mathit{\zeta} +\mathcal{\mathcal{E}}(\mathit{\zeta} ),\alpha \mathit{\zeta} ^*+\mathcal{\mathcal{E}}(\mathit{\zeta} ^*))^{\!\top} \nabla\tilde{h}^U(\mathcal{\mathcal{E}}(\mathit{\zeta} ^*))  \\
&= -\Psi(\alpha \mathit{\zeta} +\mathcal{\mathcal{E}}(\mathit{\zeta} ),\alpha \mathit{\zeta} ^*+\mathcal{\mathcal{E}}(\mathit{\zeta} ^*))^{\!\top}\sum_{i=1}^{n} v_i \nabla g_i(\mathcal{\mathcal{E}}(\mathit{\zeta} ^*)) \\
&\geq - \sum_{i=1}^{n} v_i (g_i(\mathcal{\mathcal{E}}(\mathit{\zeta} )) - g_i(\mathcal{\mathcal{E}}(\mathit{\zeta} ^*))) \\
&= \sum_{i=1}^{n} \left( -v_i g_i(\mathcal{\mathcal{E}}(\mathit{\zeta} )) + v_i g_i(\mathcal{\mathcal{E}}(\mathit{\zeta} ^*)) \right) \\
&= \sum_{i=1}^{n} -v_i g_i(\mathcal{\mathcal{E}}(\mathit{\zeta} )) \\
&\geq 0.
    \end{align*}
  When,
\begin{align*}
    \langle\ {\Psi}( {\alpha} \mathit{\zeta} + \mathcal{\mathcal{E}}(\mathit{\zeta} ), {\alpha}\mathit{\zeta^*} + \mathcal{\mathcal{E}}(\mathit{\zeta^*} )),\nabla  \tilde{h}( \mathcal{\mathcal{E}}(\mathit{\zeta^*} ))\rangle_{gH}
    = \big[
    &\Psi(\alpha \mathit{\zeta}  + \mathcal{\mathcal{E}}(\mathit{\zeta} ),\, \alpha \mathit{\zeta^*}  + \mathcal{\mathcal{E}}(\mathit{\zeta^*} ))^{\!\top}
    \nabla \tilde{h}^U(\mathcal{\mathcal{E}}(\mathit{\zeta^*} )), \\
    &\Psi(\alpha \mathit{\zeta}  + \mathcal{\mathcal{E}}(\mathit{\zeta} ),\, \alpha \mathit{\zeta^*}  + \mathcal{\mathcal{E}}(\mathit{\zeta^*} ))^{\!\top}
    \nabla \tilde{h}^L(\mathcal{\mathcal{E}}(\mathit{\zeta^*} )) 
    \big].
\end{align*}
Then we have,
\begin{equation*}
\begin{aligned}
\tilde{h}^L(\mathcal{\mathcal{E}}(\mathit{\zeta} )) - \tilde{h}^L(\mathcal{\mathcal{E}}(\mathit{\zeta} ^*)) &\geq \Psi(\alpha \mathit{\zeta} +\mathcal{\mathcal{E}}(\mathit{\zeta} ),\alpha \mathit{\zeta} ^*+\mathcal{\mathcal{E}}(\mathit{\zeta} ^*))^T \nabla\tilde{h}^U(\mathcal{\mathcal{E}}(\mathit{\zeta} ^*))  \\
&= -\Psi(\alpha \mathit{\zeta} +\mathcal{\mathcal{E}}(\mathit{\zeta} ),\alpha \mathit{\zeta} ^*+\mathcal{\mathcal{E}}(\mathit{\zeta} ^*))^T \sum_{i=1}^{n} v_i \nabla g_i(\mathcal{\mathcal{E}}(\mathit{\zeta} ^*)) \\
&\geq - \sum_{i=1}^{n} v_i (g_i(\mathcal{\mathcal{E}}(\mathit{\zeta} )) - g_i(\mathcal{\mathcal{E}}(\mathit{\zeta} ^*))) \\
&= \sum_{i=1}^{n} \left( -v_i g_i(\mathcal{\mathcal{E}}(\mathit{\zeta} )) + v_i g_i(\mathcal{\mathcal{E}}(\mathit{\zeta} ^*)) \right) \\
&= \sum_{i=1}^{n} -v_i g_i(\mathcal{\mathcal{E}}(\mathit{\zeta} )) \\
&\geq 0.
    \end{aligned}
\end{equation*}
and 
\begin{equation*}
\begin{aligned}
\tilde{h}^U(\mathcal{\mathcal{E}}(\mathit{\zeta} )) - \tilde{h}^U(\mathcal{\mathcal{E}}(\mathit{\zeta} ^*)) &\geq \Psi(\alpha \mathit{\zeta} +\mathcal{\mathcal{E}}(\mathit{\zeta} ),\alpha \mathit{\zeta} ^*+\mathcal{\mathcal{E}}(\mathit{\zeta} ^*))^T \nabla\tilde{h}^L(\mathcal{\mathcal{E}}(\mathit{\zeta} ^*))  \\
&= -\Psi(\alpha \mathit{\zeta} +\mathcal{\mathcal{E}}(\mathit{\zeta} ),\alpha \mathit{\zeta} ^*+\mathcal{\mathcal{E}}(\mathit{\zeta} ^*))^T \sum_{i=1}^{n} v_i \nabla g_i(\mathcal{\mathcal{E}}(\mathit{\zeta} ^*)) \\
&\geq - \sum_{i=1}^{n} v_i (g_i(\mathcal{\mathcal{E}}(\mathit{\zeta} )) - g_i(\mathcal{\mathcal{E}}(\mathit{\zeta} ^*))) \\
&= \sum_{i=1}^{n} \left( -v_i g_i(\mathcal{\mathcal{E}}(\mathit{\zeta} )) + v_i g_i(\mathcal{\mathcal{E}}(\mathit{\zeta} ^*)) \right) \\
&= \sum_{i=1}^{n} -v_i g_i(\mathcal{\mathcal{E}}(\mathit{\zeta} )) \\
&\geq 0.
    \end{aligned}
\end{equation*}
Thus, $\tilde{h}(\mathcal{\mathcal{E}}(\mathit{\zeta} )) \prec \tilde{h}(\mathcal{\mathcal{E}}(\mathit{\zeta}^*))$ does not hold.\\
On the other hand, if 
\begin{equation*}
    \begin{aligned}
        \min\{\tilde{h}^L(\mathcal{E}(\zeta)) - \tilde{h}^L(\mathcal{E}(\zeta^*)),~& \tilde{h}^U(\mathcal{E}(\zeta)) - \tilde{h}^U(\mathcal{E}(\zeta^*))\} = \tilde{h}^U(\mathcal{E}(\zeta)) - \tilde{h}^U(\mathcal{E}(\zeta^*)),
\end{aligned}
\end{equation*}
The remaining cases can be proven in a similar way.
   
  \end{proof}  
\end{theorem}

Next, we present the following example to illustrate and support the validity of Theorem~\ref{theorem 4.3 p4}.
\begin{example}
    Consider the following optimization problem defined on $\mathcal{S} = [\ln(2), \infty)$

\noindent where $\tilde{h} : \mathcal{S} \to \mathcal{I}(\mathbb{R})$, and $g_i : \mathcal{S} \to \mathbb{R}$

\begin{equation*}
\begin{aligned}
\text{Min } \tilde{h}(\mathit{\zeta} ) &= \left[ 4\mathit{\zeta}  - 8\ln(\mathit{\zeta} ),\, 8\mathit{\zeta}  - 16\ln(\mathit{\zeta} ) \right] \quad \text{(P1)} \\
\text{s.t. } g_1(\mathit{\zeta} ) &= \mathit{\zeta}  - 4 \leq 0 \\
g_2(\mathit{\zeta} ) &= \mathit{\zeta}  - 16 \leq 0
\end{aligned}
\end{equation*}

Let $\mathcal{\mathcal{E}} : \mathbb{R} \to \mathbb{R}$ be defined by $\mathcal{\mathcal{E}}(\mathit{\zeta} ) = e^{\mathit{\zeta} }$ and $\Psi : \mathbb{R} \times \mathbb{R} \to \mathbb{R}$ defined by $\Psi(\mathit{\zeta} ,\mathit{\delta} ) = -1$.

Then we consider the corresponding optimization problem with objective function weakly S$\mathcal{\mathcal{E}}$I with respect to $\Psi$ and constraint functions that are S$\mathcal{\mathcal{E}}$I with respect to $\Psi$, as

\begin{equation*}
   \begin{aligned}
        \text{Min } \tilde{h}(\mathcal{\mathcal{E}}(\mathit{\zeta} )) =
[4e^{\mathit{\zeta} } - 8\mathit{\zeta} ,\, 8e^{\mathit{\zeta} } - 16\mathit{\zeta} ] \quad \text{(P1$^*$)}
   \end{aligned}
\end{equation*}

\begin{equation*}
    \begin{aligned}
        \text{s.t. } g_1(\mathcal{\mathcal{E}}(\mathit{\zeta} )) &= e^{\mathit{\zeta} } - 4 \leq 0 \\
g_2(\mathcal{\mathcal{E}}(\mathit{\zeta} )) &= e^{\mathit{\zeta} } - 16 \leq 0
    \end{aligned}
\end{equation*}

The feasible set of $(\rm P1^*)$ is defined by
\begin{equation*}
    \mathcal{X} = \left\{ \mathit{\zeta}  \in \mathcal{S} : \mathit{\zeta}  \in [\ln(2), \ln(4)] \right\}.
\end{equation*}

 It can be readily verified that all the assumptions of Theorem~\ref{theorem 4.3 p4} are fulfilled.

The KKT conditions are given by

\begin{equation*}
    \begin{aligned}
        4e^{\mathit{\zeta} ^*} - 8 + v_1\, e^{\mathit{\zeta} ^*} + v_2\, e^{\mathit{\zeta} ^*} &= 0 \\
        8e^{\mathit{\zeta} ^*} - 16 + v_1\, e^{\mathit{\zeta} ^*} + v_2\, e^{\mathit{\zeta} ^*} &= 0 \\
v_1(e^{\mathit{\zeta} ^*} - 4) &= 0 \\
v_2(e^{\mathit{\zeta} ^*} - 16) &= 0
    \end{aligned}
\end{equation*}
\begin{figure}
		\centering
		\includegraphics[width=0.9\linewidth]{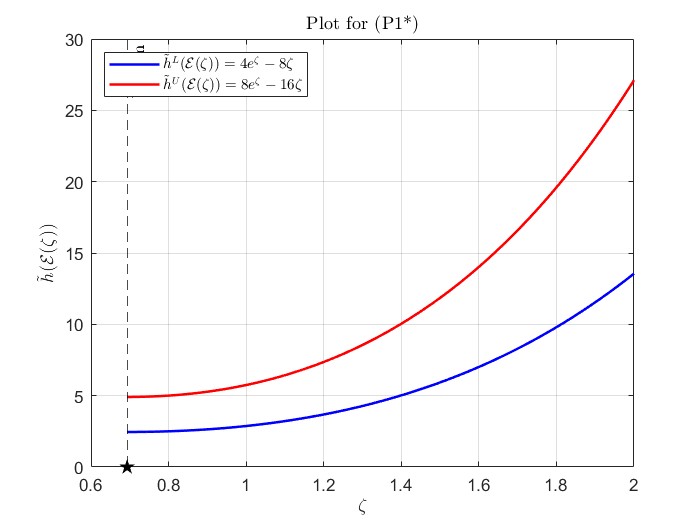}
		\label{fig:placeholder}
		\caption{$\zeta^* = ln(2)$ is a non-dominated solution of $(\rm P1^*)$}
\end{figure}

Taking $\mathit v_1 = v_2 = 0$, we get $\mathit{\zeta} ^* = \ln(2)$. Thus, by definition~\ref{definition 4.1 p4},  $\mathit{\zeta} ^* = \ln(2)$ is a non-dominated solution of $(\rm P1^*)$ which is evident from the Figure 5 also. 

\end{example}
        	\vspace{.2cm}
            
         \section{\bf Conclusion}
         In this work, we have proposed the robust notions of  strongly $\mathcal{E}$-invexity and $\mathcal{E}$-preinvexity, highlighting their distinctive characteristics and theoretical significance. To illustrate and validate these concepts, we have provided multiple illustrative examples which demonstrate their applicability and usefulness.
 This paper introduces and analyzes new classes of interval-valued functions—namely, SLU$\mathcal{E}$I, SLU$\mathcal{E}$P and PSLU$\mathcal{E}$P functions—under the LU ordering framework. By extending classical concepts of strong invexity and preinvexity to the interval setting, we have provided a richer theoretical foundation for handling uncertainty within nonlinear programming models.

Key structural properties were established, and explicit relationships among these new function classes and existing forms of generalized convexity were derived. Illustrative examples and counterexamples have been presented to clarify the scope and limitations of the proposed classes.

As a significant application, we formulated a nonlinear programming problem involving SLU$\mathcal{E}$P functions and the sufficiency of KKT optimality conditions . These results guarantee global optimality and non-dominated solutions under the generalized preinvexity and invexity assumptions, thereby extending classical optimality theory to more complex interval-valued settings on Riemannian manifolds.

Overall, this study advances the theory of interval-valued optimization and provides a robust framework for future research on generalized convexity and its applications in uncertain decision-making environments.

      \end{document}